\documentclass[12pt,twoside]{article}

\usepackage{xr}
    \usepackage{graphicx}
    \usepackage{styleset}
    \usepackage{macros}
    \let\subsubsection\subparagraph

    \title  {Instantons and some
      concordance invariants of knots}

    \author {P. B. Kronheimer and T. S. Mrowka%
      \thanks{%
        The work of the first author was supported by the National
        Science Foundation through NSF grants
        DMS-1405652 and DMS-1707924. The work of the second author was supported by
        NSF grants DMS-1406348 and DMS-1808794, and by a grant from the Simons Foundation,
        grant number 503559 TSM.}}

    \address {Harvard University, Cambridge MA 02138 \\
              Massachusetts Institute of Technology, Cambridge MA 02139}

\begin{document}

\maketitle

\begin{abstract}
 Concordance invariants of knots are derived from the
instanton homology groups with local coefficients, as introduced in
earlier work of the authors. These concordance invariants 
include a $1$-parameter
family of homomorphisms $\scrf_{r}$, from the knot concordance group to
$\R$. Prima facie, these concordance invariants have 
the potential to provide independent bounds on the genus
and number of double points for immersed surfaces with boundary a
given knot.

\end{abstract}

\tableofcontents

\section{Introduction}

\subsection*{Concordance invariants from instanton homology}

For a knot $K$ in a closed, oriented 3-manifold $Y$, 
the authors' earlier papers have introduced instanton
homology groups $\Isharp(Y,K)$, and reduced instanton homology groups
$\Inat(K)$.
Several variants of the basic construction are possible.
In particular, in \cite{KM-ibn1}, a version of instanton
homology with local coefficients $\Isharp(Y,K ; \Gamma)$ is
constructed. Here $\Gamma$ is a local system of coefficients over an
appropriate configuration space of connections associated to
$(Y,K)$. It is a system of free rank-1 $\cR$-modules,
where $\cR$ is the ring of finite Laurent series in $4$ variables over the
field of two elements:
\begin{equation}\label{eq:R-def}
              \cR = \F_{2}[T_{0}^{\pm 1} , T_{1}^{\pm 1}, T_{2}^{\pm 1}, T_{3}^{\pm 1} ].
\end{equation}
Given a base change $\s : \cR\to\cS$, we write $\Gamma_{\s}$ for the
local system $\Gamma \otimes_{\s}\cS$, and we can construct the
groups $\Isharp(Y,K; \Gamma_{\s})$. If $\s$ satisfies the
condition $\s(T_{0})=\s(T_{1})$, then the reduced groups
$\Inat(Y,K;\Gamma_{\s})$ are also defined. In the case of the unknot,
the reduced group is a free $\cS$-module of rank $1$, while the
unreduced group is free of rank $2$. For details we refer to
\cite{KM-ibn1} and the references therein, though we shall summarize
some features in Section~\ref{sec:review}.

The purpose of this paper is to explore how these instanton homology
groups with local coefficients
give rise to (potentially) new concordance invariants of knots. In its ``raw''
form, the first such invariant associates to a knot $K\subset S^{3}$ a fractional
ideal
\[
            \zsharp(K) \subset \Frac(\cR).
\]
There is also a version of this construction using the reduced version
$\Inat(K; \Gamma)$ in place of $\Isharp(K;\Gamma)$ giving rise to a
potentially different ideal,
\[
            \znat_{\s}(K) \subset \Frac(\cS).
\]   
for any base-change $\s$ with $\s(T_{0}) = \s(T_{1})$.
Other invariants can then be derived from this construction. In
particular if we have a base-change $\s: \cR \to \cS$ where $\cS$ is a
valuation ring, and if $\s(T_{0})=\s(T_{1})$,
then one may construct a \emph{homomorphism}
\[
           \scrf_{\s} : \Conc \to \mathrm{Val}(\cS),
\]
where $\Conc$ is the concordance group of knots and
$\mathrm{Val}(\cS)$ is the valuation group of $\cS$. (See
Section~\ref{subsec:conc-hom} below.)

Similar
constructions, using Heegaard Floer homology rather than instanton
homology, 
have been made earlier by Ozsvath
and Szabo in \cite{Upsilon} and by Alishahi and Eftekhary in
\cite{Alishahi-Eftekhary}. In the context of gauge theory, similar
constructions occur in \cite{KM-singular} and
\cite{KM-s-invariant}. Like their earlier relatives, the concordance
invariants $\scrf_{\s}$ defined in this paper provide lower bounds for
the slice genus of a knot.

An intriguing  feature of this construction in the instanton case is that, for suitably
chosen $\s$, the concordance invariant may provide \emph{independent}
control of the genus and number of double points for normally immersed surfaces in
the ball. (By \emph{normally immersed} we shall mean that the the only
self-intersection points of the immersed surface are transverse double
points.) More specifically, for each $r\in [0,1]$, we can define a
homomorphism
\[
            \scrf_{r} : \Conc \to \R
\]
with the following property. Suppose $K$ is a knot in $S^{3}$ bounding
a normally immersed, connected,
oriented,  surface $S\subset B^{4}$. 
Let $\gen(S)$ be its genus, and $\dplus(S)$
the number of positive  double points. Then the
concordance invariant $\scrf_{r}$ satisfies the inequality
\begin{equation}\label{eq:g-delta-r}
         \gen(S) + r \dplus(S) \ge \scrf_{r}(K).
\end{equation}    
The authors'  invariant $(1/2)s^{\sharp}(K)$ from \cite{KM-s-invariant} satisfies an
inequality of this sort with $r=1$, as does the Ozsvath-Szabo
$\tau$-invariant \cite{OS-tau}. 

A priori, the new invariants with $r<1$
potentially constrain the ability to ``trade handles for
double-points'' in immersed surfaces. Since any knot bounds a
normally immersed disk, it is clear from the shape of
\eqref{eq:g-delta-r} 
that $\scrf_{r}(K) \to 0$ as $r\to 0$. So for small $r$ the inequality
contains essentially no information about the genus of $S$. By
considering a limiting case, we shall arrive also at a concordance
homomorphism
\begin{equation}\label{eq:scrfstar}
        \scrf_{*} : \Conc \to \R
\end{equation}
with the property that, for a normally immersed disk $S$ in $B^{4}$ with
boundary $K$,  we have
\[
       \dplus(S) \ge \scrf_{*}(K).
\]
The minimal number of crossings in a normally immersed disk -- without
concern for the signs of the crossings -- is sometimes called the 
\emph{4-dimensional clasp number}
or \emph{4-dimensional crossing number} of the knot $K$, and
is often written $c_{*}(K)$. It follows that $|\scrf_{*}(K)|$ is a lower
bound for the 4-dimensional clasp number of a knot $K$:
\[
      c_{*}(K) \ge | \scrf_{*}(K) |.
\]
On the other hand, $|\scrf_{*}(K) |$ may not be a lower bound for the
slice genus of $K$.

\begin{remarks}
As mentioned above, the association of an ideal $\znat(K)$ to a knot $K$ is formally similar
also to the construction used by Alishahi and Eftekhary in
\cite{Alishahi-Eftekhary}, which is based on a variant of
Heegaard-Floer homology for knots. Our invariants $\scrf_{r}$, obtained
from $\znat(K)$ by base-change to a suitable valuation ring, are similarly
related in a formal way to the invariants $\Upsilon(t)$ defined by Ozsvath, Stipsicz and Szabo
in \cite{Upsilon}, which one can derive from the Alishahi-Eftekhary
invariants by a similar base change.

Although the authors hope to return to this in the future, the present
paper contains no complete calculations of $\znat(K)$ or $\scrf_{r}(K)$ for any
non-trivial knots except the trefoil. For the simplest knots, such as 2-stranded torus
knots, twist knots, and some small pretzel knots including the knot
$7_{4}$ in the Rolfsen table, calculations of
these invariants can be made based on just the formal properties that
we establish. The results of some of these calculations are
summarized at the end of this paper, but not included in detail,
though we do include an alternative calculation for  the trefoil to
illustrate aspects of the gauge theory. 
The constructions in \cite{Alishahi-Eftekhary} also
associate to a knot $K$ a  module $\mathbb{A}(K)$ isomorphic to a
monomial ideal in a ring of polynomials in two variables over a field
of characteristic $2$. In the simplest cases that the authors have
calculated, $\znat(K)$ agrees with $\mathbb{A}(K)$ after a
base-change, but this appears to be a consequence of the fact that the
two invariants share similar formal properties. For the torus knot
$T(3,4)$, the authors believe that the two invariants are
different. (See section~\ref{subsec:further}.)
 One part of the difficulty in calculation arises from the
fact we are working in characteristic $2$, where there is less prior
work on instanton homology, though see \cite{Scaduto-Stoffregen, Scaduto}.
\end{remarks}

\subsection*{Non-orientable surfaces}

There is a further formal similarity between the families of concordance  homomorphisms
$\scrf_{\s}$ defined here and the invariants $\Upsilon(t)$,
($0<t<2$),  defined by Ozsv\'ath, Stipsicz and Szab\'o in
\cite{Upsilon}. Like $\scrf_{\s}$, each $\Upsilon(t)$ is a
homomorphism from the concordance group to $\R$. But in addition, 
it is shown in \cite{Upsilon-2} that for the
special case $t=1$, the invariant $\Upsilon(t)$ constrains the
topology of \emph{non-orientable} embedded surfaces $S\subset B^{4}$ with
boundary $K$, by an inequality
\[
             b_{1}(S) -\frac{1}{2}\nu(S) \ge -2 \Upsilon_{K}(1),
         \]
where $\nu$ is the degree of the normal bundle of $S$ relative to the
zero framing of the boundary.
We shall see that certain specializations of the construction of
$\scrf_{\s}$ lead to concordance homomorphisms with the same property.

\subsection*{Crossing changes}

We shall also describe the behavior of $\Isharp(K;\Gamma)$ under
crossing-changes of $K$. The rank of
$\Isharp(K;\Gamma)$ over $\cR$ is unchanged by crossing-changes,
and only the torsion is affected. This is the same behavior as is
described for $\Isharp(K;\Gamma_{o})$ in the authors' earlier paper
\cite{KM-s-invariant}, which in turn rested on \cite{KM-singular} and
\cite{Obstruction}. 
 
Corresponding results for Heegaard-Floer homology are proved in
\cite{OSS-Grid} and \cite{Alishahi-Eftekhary}, and  
similar results for Bar-Natan homology and Lee homology were proved by Alishahi
and  Alishahi-Dowlin in
\cite{Alishahi, Alishahi-Dowlin}.
As in \cite{Alishahi, Alishahi-Dowlin,Alishahi-Eftekhary},
one can exploit the crossing-change behavior to show that the
torsion part of $\Isharp(K;\Gamma)$ gives rise to a bound on the
number of crossing changes needed to unknot $K$.

\section{Review of instanton homology with local coefficients}
\label{sec:review}

\subsection*{The basic construction}

We briefly recall some of the features of the instanton homology
groups from \cite{KM-ibn1}. Let $K$ be a link in a closed oriented
3-manifold $Y$, let $y_{0}\in Y$ be a framed basepoint and $B(y_{0})$ a
standard ball, disjoint from $K$. Let $\theta$ be a standard
theta-graph embedded in $B(y_{0})$, and let
\[
          K^{\sharp} = K \cup \theta
\]
be the union, regarded as a web embedded in $Y$. We can equip $Y$ with
the structure of an orbifold $\check Y$ whose singular set is $K^{\sharp}$ and
whose local stabilizers are $\Z/2$ along all edges of the web. There is
then an associated space $\bonf^{\sharp}(Y,K)$ which parametrizes
isomorphism classes of orbifold $\SO(3)$ connections on $\check Y$
equipped with a lift to $\SU(2)$ on the complement of the singular
set. The instanton homology $\Isharp(Y,K)$ with coefficients $\F_{2}$
is constructed as the Morse homology of a perturbed Chern-Simons
functional on $\bonf^{\sharp}(Y,K)$.

To define a system of local coefficients, one starts by constructing
four maps
\[
       h = (h_{0}, h_{1}, h_{2}, h_{3}) : \bonf^{\sharp}(Y,K)\to (\R/\Z)^{4}.
\]
The component $h_{0}$ is defined using the holonomy of a connection
along all the components of  the link $K$, while $h_{1}, h_{2}, h_{3}$ are defined using the
holonomy along the edges of $\theta$. See \cite{KM-ibn1}. The local
system $\Gamma$ over $\bonf^{\sharp}(Y,K)$ is defined as the pull-back
via $h$ of a tautological local system over $(\R/\Z)^{4}$ whose fiber is a
free rank-1 module over the group ring $\cR=\F_{2}[\Z^{4}]$. Formal
variables $T_{i}$ are introduced so as to write $\cR$ as the ring of
finite Laurent series \eqref{eq:R-def}.

\begin{definition}[\protect{\cite[section~\ref{1-subsec:local-system}]{KM-ibn1}}]
    The instanton homology group of $(Y,K)$, denoted
    $\Isharp(Y,K;\Gamma)$, is the Floer homology group constructed
    from the perturbed Chern-Simons functional on
    $\bonf^{\sharp}(Y,K)$ with coefficients in the local system
    $\Gamma$. For any ring homomorphism of commutative rings,
    $\s:\cR\to \cS$, we write $\Gamma_{\s}$ for the local system
    $\Gamma \otimes_{\s}\cS$, and $\Isharp(Y,K;\Gamma_{\s})$ for the
    instanton homology.
\end{definition}

Despite the appearance of the definition of $h_{0}$, a careful
examination of the local system shows that the orientation of the link
$K$ plays no role.

\subsection*{A variant with non-zero Stiefel-Whitney class}

We also recall from \cite{KM-ibn1} that given closed 1-manifold
$\omega\subset Y$ disjoint from $K^{\sharp}$, there is a variant of
$\Isharp(Y,K;\Gamma_{\s})$ constructed from $\SO(3)$ connections whose
Stiefel-Whitney class is dual to $\omega$. More precisely, the space
$\bonf^{\sharp}(Y,K)_{\omega}$ is defined as a space of orbifold
$\SO(3)$ connections on $\check Y$ together with a lift to $\SU(2)$ on
the complement of $K^{\sharp}\cup \omega$ and such that the
obstruction to extending the lift across $\omega$ is $-1$. The local
system $\Gamma_{\s}$ can be defined on $\bonf^{\sharp}(Y,K)_{\omega}$
for any $\s$, and we have instanton homology groups
\[
           \Isharp(Y, K;\Gamma_{\s})_{\omega}.
\]

Rather than being a closed $1$-manifold in the complement of the web, the locus $\omega$ can also
be allowed to have components which are arcs with end-points on the
link $K$. When $\omega$ has this form, the holonomy map $h_{0}$ can no
longer be constructed using holonomy along $K$, and the local system
$\Gamma$ is no longer defined. However, if $\s : \cR\to \cS$ is a base
change with $\s(T_{0})=1$, then $h_{0}$ plays no role in the
definition of the local system $\Gamma_{\s}$. For such $\omega$, we
may therefore define $\Isharp(Y, K ;\Gamma_{\s})_{\omega}$ whenever $\s(T_{0})=1$.

\subsection*{Functoriality for embedded cobordisms}

Having briefly reviewed the main features of the instanton homology groups
$\Isharp(Y, K ;\Gamma)$, we now turn to their functorial properties.
Let $(X,S)$ be a cobordism from a pair $(Y_{0}, K_{0})$ to a pair
$(Y_{1}, K_{1})$.We require as usual that $X$ is
oriented so that
\[
      \partial X = -Y_{0} + Y_{1},
\]
and in this section we will also require that $S$ is an oriented
cobordism betweeen oriented links:
\[
       \partial S = -K_{0} + K_{1}.
\]
(This condition of orientability will be dropped later.)
Because we wish to consider the instanton homology $\Isharp$, we
require standard embedded balls $B(y_{0})$ and $B(y_{1})$ at 
framed base-points $y_{0}$ and $y_{1}$, and an embedded $[0,1]\times
B^{3}$ joining these in $X$
(see \cite[Section~\ref{1-subsec:functoriality}]{KM-ibn1}). We
always require $S$ to be disjoint from these. However, we will usually
not indicate them in our notation. 
The functoriality of instanton homology means that $S$ gives rise to a
map of $\cR$-modules,
\[
            \Isharp(X,S;\Gamma) : \Isharp(Y_{0}, K_{0} ; \Gamma) \to
            \Isharp(Y_{1}, K_{1};\Gamma).
\]

This basic construction can be extended in various ways. First,
without essential change, we can pass to a local system of $\cS$-modules
$\Gamma_{\s}$ in place of $\Gamma$, by a base change $\s :
\cR\to\cS$. Second, we can consider functoriality for the homology
groups  modified by a codimension-2 representative $\omega$, as
described in the previous paragraphs. Given a closed 1-manifolds
$\omega_{i}\subset Y_{i}$ disjoint for $K_{i}^{\sharp}$ for $i=0,1$,
and corresponding homology groups $\Isharp(Y_{i}, K_{i};
\Gamma_{\s})_{\omega_{i}}$. Given a cobordism $(X,S)$ as before and
also a 2-dimensional submanifold $\omega$ -- a cobordism from
$\omega_{0}$ to $\omega_{1}$, disjoint from $S$ -- then one obtains a
map
\[
              \Isharp(X,S;\Gamma)_{\omega} : \Isharp(Y_{0}, K_{0} ; \Gamma)_{\omega_{0}} \to
            \Isharp(Y_{1}, K_{1};\Gamma)_{\omega_{1}}.        
\]
As discussed in \cite{KM-ibn1}, one may allow $\omega$ to have
transverse intersections with the interior of $S$ in $X$. Furthermore,
in the special case that $\s(T_{0})=1$, one may allow $\omega$ to be a
manifold with corners whose boundary pieces are $\omega_{0}$ and
$\omega_{1}$ together with arcs and circles on $S$.

Since $X$ will usually be fixed, and coefficients $\Gamma$ or
$\Gamma_{\sigma}$ are
understood, we may abbreviate the notation and just write, for example,
\[
          \Isharp(S) =  \Isharp(X,S;\Gamma_{\s}).
\]

\subsection*{Immersed cobordisms}

 We may allow the surface $S$ to be \emph{normally
  immersed} (immersed with normal crossings) in
$X$. To extend the definition of
$\Isharp(X,S;\Gamma_{\sigma})_{\omega}$ to this case, we first transform $S$ to an embedded surface by a
blow-up construction: at each normal crossing, we replace $S$ by its proper transform $\tilde S$
in $\tilde X = X\# \bar{\CP}^{2}$.    Following the convention about the immersed case
that is captured by the
formula 
To define this map in the case that $S$ is normally immersed rather
than embedded, we first transform $S$ to an embedded surface by a
blow-up construction. In the case of just a single double point, one
then defines $\Isharp(S)$ in terms of $\Isharp(\tilde S)$ by the rule:
\begin{equation}\label{eq:blow-up-rule}
         \Isharp(X, S ; \Gamma_{\s})_{\omega} = \Isharp(\tilde X, \tilde S
         ; \Gamma_{\s})_{\omega}
        +  \Isharp(\tilde X, \tilde S ; \Gamma_{\s})_{\omega+\epsilon},
\end{equation}
where $\epsilon$ is the standard $2$-sphere representing the generator
of $H^{2}(\bar{\CP}^{2})$.  For more than one double-point, one
applies this construction repeatedly. See
 \cite[Section~\ref{1-subsec:double-points}]{KM-ibn1} for details.

The reason for including both terms on the right hand side of
\eqref{eq:blow-up-rule} is explained in \cite{KM-ibn1}, and is
appropriate when we wish to allow $\omega$ to have boundary along
$S$. If $\omega$ and $\omega'$ are two surfaces with boundary on $S$,
and if they are isotopic by an isotopy in which $\partial \omega$
sweeps over a double-point of $S$ along one branch of the immersed
surface, then the classes $[\omega']$ is homologous to
$[\omega+\epsilon]$. Both terms are therefore needed if we wish to
have a definition that is invariant under this isotopy. That said, if
we allow only the more restricted representatives $\omega$ whose
boundaries do not include arcs or circles on $S$, then this symmetry
between the two terms is no longer needed. Following \cite{KM-ibn1},
we can therefore change the definition of the functor when applied to
normally immersed surfaces. For any $\xi\in \cS$ we may define a modified
functor $\Isharp_{\xi}$ by altering \eqref{eq:blow-up-rule} to:
\begin{equation}\label{eq:blow-up-rule-xi}
         \Isharp_{\xi}(X, S ; \Gamma_{\s})_{\omega} = \Isharp_{\xi}(\tilde X, \tilde S
         ; \Gamma_{\s})_{\omega}
        +  \xi \,\Isharp_{\xi}(\tilde X, \tilde S ; \Gamma_{\s})_{\omega+\epsilon}.
\end{equation}
When applied to embedded surfaces, the functors $\Isharp_{\xi}$ and $\Isharp$ are equal.

\subsection*{Surfaces with dots}

The functoriality can be extended by allowing the morphisms to be
surfaces $S$ with additional \emph{decoration by  dots}: a dot is an interior
point $q$ of $S$ together with an orientation of $T_{q}S$. Thus
given $S$ and a dot $q$, there is an operator
\[
       \Isharp(S, q) : \Isharp(Y_{0}, K_{0} ; \Gamma) \to
            \Isharp(Y_{1}, K_{1};\Gamma).
\]
To define this extension, following section
\ref{1-subsec:surfaces-with-dots} of \cite{KM-ibn1}, it is sufficient
to treat the case that $S$ is a trivial cobordism (a cylinder), in
which case we are seeking to define an operator
\[
             \Lambda_{q} :  \Isharp(Y, K; \Gamma)\to  \Isharp(Y, K; \Gamma).
\]
This operator, which has even degree for the mod 2 grading, is defined
in \cite{KM-ibn1}, where it is shown to satisfy a relation
\[
                                    \Lambda^{2}_{q} + P \, \Lambda_{q} + Q = 0,
\]
where $P$ and $Q$ are elements of $\cR$ given by
\begin{equation}\label{eq:PQ-formulae}
\begin{aligned}
        P &= T_{1}T_{2}T_{3} + T_{1}T_{2}^{-1}T_{3}^{-1} +
        T_{2}T_{3}^{-1}T_{1}^{-1} + T_{3}T_{1}^{-1}T_{2}^{-1}
\end{aligned}
\end{equation}
and
\[
    Q = \sum_{j=0}^{3} (T^{2}_{j} + T_{j}^{-2}).
\]

\section{The ideal of a knot}
\label{subsec:gd-bounds}

\subsection*{Modifying surfaces in the orientable case}

There are standard ways in which an embedded or immersed surface $S$
can be modified by local operations to produce a new surface, and
there are formulae proved in \cite{KM-ibn1} for how the map 
$\Isharp(S)$ may be changed such modifications.
For the following definition, we refer to \cite{KM-ibn1, KM-singular}
for a description of the twist moves and finger move.
The ``internal 1-handle addition'' is illustrated in
Figure~\ref{fig:handle-addition}. To describe it in words, 
points $p$ and $q$ lie on two disks
in $S$, both of which are standardly embedded in a ball
$B^{4}\subset X$. The new surface $S^{*}$ is obtained by replacing the
two oriented disks with a standard oriented annulus. A special case of
this operation is an internal connect sum with a standard torus.

\begin{definition}\label{def:equiv}
Let $(X,S)$ and $(X,S^{*})$ be cobordisms between the same pairs (so only
the surface $S$ has changed). We continue to suppose that $S$ and
$S^{*}$ are oriented. We will say that $S^{*}$ is
\emph{c-equivalent} to $S$ if $S^{*}$ can be obtained from $S$ by a
sequence of moves, each of which is one of the following
or its inverse:
\begin{itemize}
\item an ambient isotopy relative to the boundary;
\item  introducing a double point by a twist  move, either positive or negative;
\item  introducing two new double points by a finger move;
\item  an oriented internal 1-handle addition connecting two points in the same
    connected component of the surface.
\end{itemize}
\end{definition}

\begin{remark}
If $X$ is simply connected and
$S\subset X$ is connected, then $c$-equivalence is the same as the
equivalence relation generated by homotopy relative to the boundary
together with ``stabilization'' by internal connected sum with
$T^{2}$. For example, given two classical knots $K_{0}$ and $K_{2}$, any two
connected, oriented cobordisms $S$ and $S^{*}$ joining them in
$[0,1]\times S^{3}$ are $c$-equivalent.
\end{remark}

\begin{figure}
    \begin{center}
        \includegraphics[scale=0.3]{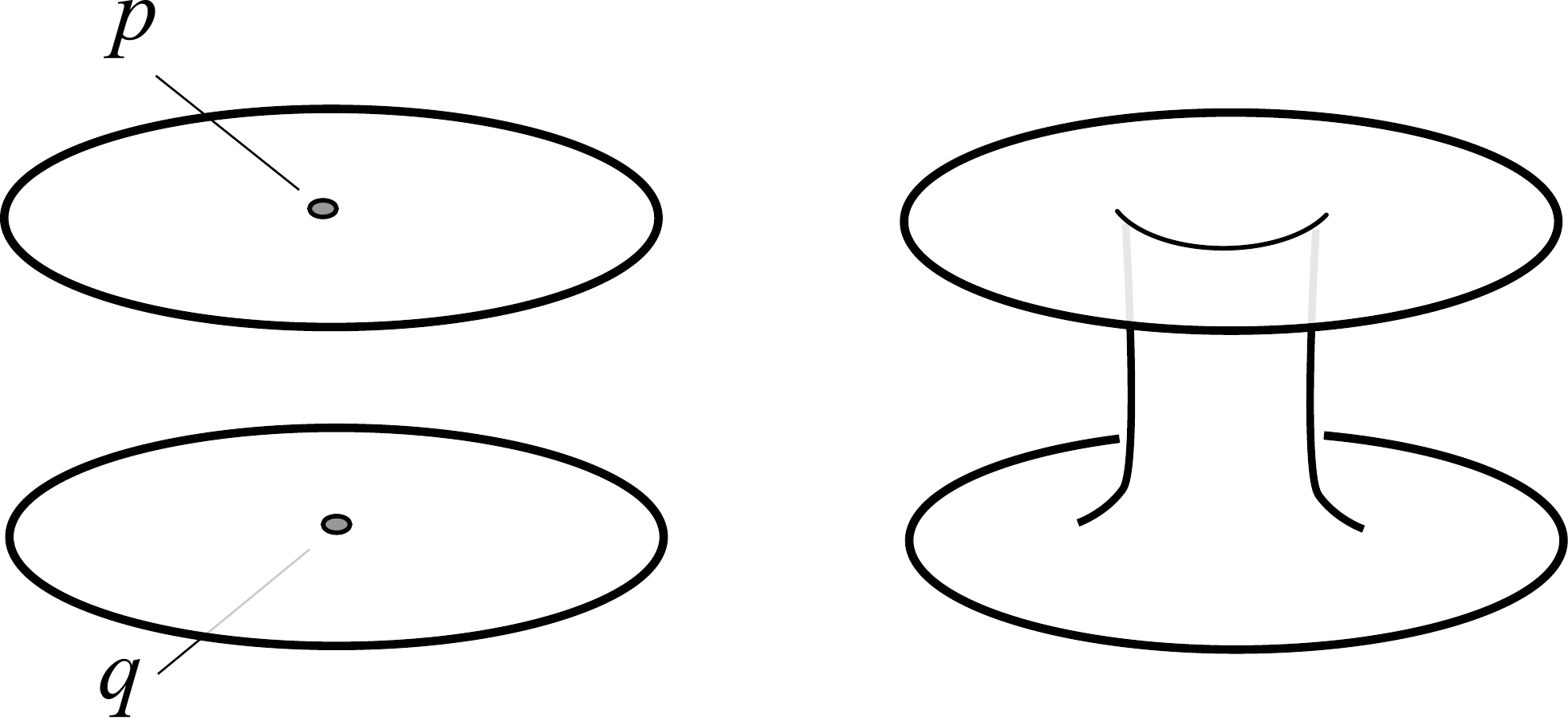}
    \end{center}
    \caption{\label{fig:handle-addition}
    The internal addition of a handle. The surfaces are embedded in a
    standard $4$-ball.}
\end{figure}

Formulae for how $\Isharp(S)$ changes when $S$ is changed by
a twist move or a finger move were given in
\cite[Proposition~\ref{1-prop:twist-and-finger-formulae}]{KM-ibn1}. The
lemma below summarizes these, and also provides a formula for internal
$1$-handle additions (a ``neck-cutting relation'' which generalizes
the formula for  the internal connected sum with a torus
\cite[Lemma~\ref{1-lem:torus-sum}]{KM-ibn1}).

\begin{lemma}\label{lem:1-handle-add}
\begin{enumerate}
\item If $S'$ is obtained from the oriented immersed cobordism $S$ by
    either a finger move or a positive twist move, then
\[
     \Isharp(S') =\sigma( \V) \Isharp(S)
\]
where
\[
              \V = P + T_{0}^{2} + T_{0}^{-2} \in \cR.
\]
\item If $S'$ is obtained from $S$ by
    either a negative twist move, the $\Isharp(S') = \Isharp(S)$.
\item
    If $S'$ is obtained from the oriented cobordism $S$ by an internal $1$-handle addition
    connecting two points $p$ and $q$ of $S$, 
     \[
            \Isharp(S') =  \Isharp(S, p) + \Isharp(S,q) + P\, \Isharp(S),
     \]
    where the notation $(S,p)$ means the surface $S$ decorated with a
    dot at $p$ using the orientation of $S$. In particular, if $p$ and
    $q$ are on the same component of $S$, then
\[
  \Isharp(S') =  P\, \Isharp(S).
\]
\end{enumerate}
\end{lemma}

\begin{proof}
    It remains to prove the result for the internal $1$-handle
    addition. Using excision, we can reduce the general case to the case that
    $S$ is the trivial cobordism from the 2-component unlink to
    itself. In this case, $S'$ can be described as a composite of the
    ``pants'' and ``copants'' cobordisms, and $\Isharp(S')$ can therefore be
    calculated using the results from
    section~\ref{1-subsec:pants-copants} of \cite{KM-ibn1}.
\end{proof}

Two topological quantities associated to an oriented surface $S$ will be our focus
here. The first is the number of positive double points, which we
write as $\dplus(S)$. The second is the genus of $S$. In the context of cobordisms, it
is convenient to make an adjustment to the genus:

\begin{definition}\label{def:ga}
    For a cobordism $S$ from a link $K_{0}$ to a link $K_{1}$, we
    define the adjusted genus to be the quantity
    \begin{equation}\label{eq:adjusted-genus}
        \ga(S) = \bigl( -\chi(S) + c_{+}(S) - c_{-}(S)\bigr)/2,
    \end{equation}
    where $c_{+}$ and $c_{-}$ are the number of components of the
    outgoing and incoming ends of of the cobordism (the number of
    components of $K_{1}$ and $K_{0}$ respectively).
\end{definition}

\begin{remark}
If $K_{1}$ is a knot, then
$\ga(S)$ coincides with the usual genus of the surface, which is why
there is no risk of confusion in using the same notation. The
advantage of the adjusted genus is that it is additive for composite cobordisms.
\end{remark}

With these definitions out of the way, we can state the main result
from which the remainder of our conclusions will be derived.

\begin{proposition}\label{prop:c-equiv}
    Suppose $S$ and $S^{*}$ are c-equivalent. Let $\ga(S)$ and
    $\dplus(S)$ denote the adjusted genus and the number of positive double
    points. Then there exist $n$ and $m$ large enough such that
    \[
                 P^{\ga(S^{*}) + n} \V^{\dplus(S^{*})+ m}  \Isharp(S)  = 
                   P^{\ga(S) + n} \V^{\dplus(S)+ m} \Isharp(S^{*}).
    \]
\end{proposition}

\begin{remark}
    If $\Isharp(K_{1} ; \Gamma)$ is torsion-free and $\ga(S), \ga(S^{*})\ge 0$, then there is no need
    for $n$ and $m$ in this proposition: the stated equality  holds
    only if it already holds with $m=n=0$. 
\end{remark}

\begin{proof}[Proof of the Proposition]
    The essential calculations here are to show that 
   the stated equality holds if $S^{*}$ is obtained from $S$ by just one
    of the moves in Definition~\ref{def:equiv} or their inverses. For
    such a single move, the result
    follow in each case from Lemma~\ref{lem:1-handle-add} (bearing in
    mind that Definition~\ref{def:equiv} requires that $p$ and $q$ are
    in the same component  of $S$, in the case for an internal
    $1$-handle addition).
    For more than one move, the result
    follows by induction.
\end{proof}

\subsection*{Definition of the ideal}

Let $\Frac(\cR)$ denote the field of fractions of $\cR$. If $M$ and $N$ are
submodules of an $\cR$-module $W$, we write $[N:M]$ for the generalized
module quotient,
\[
            [N : M] = \{ \, a/b \in \Frac(\cR) \mid aM \subset b N \}.
\]
This is an $\cR$-submodule of $\Frac(\cR)$. In this definition, if the
modules might have torsion, we should allow $a/b$ to be a fraction
that is not expressed in lowest terms.

\begin{definition}\label{def:zXS-def}
    Given a cobordism
$(X, S)$ as above from $(Y_{0}, K_{0})$ to $(Y_{1}, K_{1})$, we define
the $\cR$-submodule  \[ \zeta(X,S) \subset \Frac(\cR) \] to be
$             \zeta(X,S) = [N : M ]$,
where $N=  \Isharp(Y_{1}, K_{1};\Gamma)$ and $M \subset N$ is the
image of the $\cR$-module homomorphism $ \Isharp(S ;\Gamma)$:
\[
       M = \im \Isharp(S ;\Gamma).
\]
Further, if $\ga$ and $\dplus$ are the adjusted genus and number of positive double
    points in $S$, we define
    \[
    \begin{aligned}
        \zsharp(X,S) &= P^{\ga} \V^{\dplus} \zeta(X,S) \\
        &\subset \Frac(\cR).
    \end{aligned}
    \]
   Note that $1\in \zeta(X,S)$, so $\zsharp(X,S)$ is always a non-zero submodule of $\Frac(\cR)$.
\end{definition}

\begin{lemma}\label{lem:c-equivalent-equal}
    If $S^{*}\subset X$ is c-equivalent to $S$, then the $\cR$-submodules
    $\zsharp(X,S)$ and $\zsharp(X,S^{*})$ in $\Frac(\cR)$ are equal.
\end{lemma}

\begin{proof}
    This follows from the definition and Proposition~\ref{prop:c-equiv}.
\end{proof}

\begin{corollary}\label{cor:in-zXS}
    If $S$ is c-equivalent to a surface $S^{*}$ with adjusted genus $\ga^{*}$ and
    $\dplus^{*}$ positive double points (and any number of negative
    double points), then
   \[
                      P^{\ga^{*}}\V^{\dplus^{*}} \in \zsharp(X,S).
    \]
\end{corollary}

\begin{proof}
     Given Proposition~\ref{prop:c-equiv}, this again follows by
     checking the definitions. To do this,
     let $M$ be the image of $\Isharp(S)$ and let $M_{*}$ be the image
     of $\Isharp(S^{*})$. Since $M_{*}$ is
    certainly contained in $\Isharp(Y_{1}, K_{1};\Gamma)$, we evidently have
     \[
              r P^{\ga} \V^{\dplus} M_{*} \subset  r  P^{\ga} \V^{\dplus}
              \Isharp(Y_{1}, K_{1};\Gamma).
     \]
    From Proposition~\ref{prop:c-equiv}, with $r=P^{n}\V^{m}$, we have
    \[
            r P^{\ga} \V^{\dplus} M_{*} =  r P^{\ga^{*}} \V^{\dplus_{*}} M ,
    \]
    so the previous inclusion also says
    \[
                    r P^{\ga^{*}} \V^{\dplus_{*}} M \subset   r  P^{\ga} \V^{\dplus}
              \Isharp(Y_{1}, K_{1};\Gamma).
    \]
     From the definition of $\zeta(X,S)$, this last inclusion means
     \[
               P^{\ga^{*}-\ga}\V^{\dplus_{*} - \dplus} \in \zeta(X,S) ,
     \]
     which is equivalent to
     \[
            P^{\ga^{*}}\V^{\dplus_{*}} \in \zsharp(X,S).
    \]
\end{proof}

\begin{lemma}\label{lem:c-product-R}
    If\/ $X$ is a product, $[0,1]\times Y$, and $S$ is c-equivalent to a
    product $[0,1]\times K$, then $\zsharp(X,S) = \cR$. 
\end{lemma}

\begin{proof}
    Because of Lemma~\ref{lem:c-equivalent-equal}, we may as well
    assume $S$ is  product. In this case, the map is the identity, so
    $\zeta(X,S)= \cR$ by construction. On the other hand, $\ga$ and
    $\dplus$ are both zero for a product cobordism, so $\zsharp(X,S)=\zeta(X,S)$.
\end{proof}

\begin{lemma}\label{lem:composite-z}
    Suppose $(X_{i}, S_{i})$ is a cobordism from $(Y_{i-1}, K_{i-1})$
    to $(Y_{i}, K_{i})$, for $i=1,2$. Let $(X,S)$ be the composite
    cobordism. Then
   \[
                   \zsharp(X_{1}, S_{1})\,\zsharp(X_{2}, S_{2}) \subset \zsharp(X,S).
   \]
\end{lemma}

\begin{proof}
    The terms $\ga(S)$ and $\dplus(S)$ are both additive, so
    the assertion is equivalent to
    \[
         \zeta(X_{1}, S_{1}) \zeta(X_{2}, S_{2}) \subset \zeta(X,S).
    \]
   For brevity, write $N_{i}= \Isharp(Y_{i}, K_{i};\Gamma)$ for
   $i=0,1,2$. If $a_{1}/b_{1} \in \zeta(X_{1}, S_{1})$ and
   $a_{2}/b_{2}\in \zeta(X_{2}, S_{2})$, then
   \[
   \begin{aligned}
       (a_{1}a_{2}) \im (\Isharp(S)) &= a_{2} \Isharp(S_{2})\bigl (
       a_{1}\im (\Isharp(S_{1})\bigr) \\
       &\subset a_{2} \Isharp(S_{2})(b_{1} N_{1})\\
       &= b_{1} a_{2} \im\Isharp(S_{2}) \\
       &\subset b_{1}b_{2} N_{2}.
   \end{aligned}
   \]
   So $(a_{1}a_{2}) / (b_{1} b_{2}) \in \zeta(X,S)$ as required.
\end{proof}

\begin{corollary}\label{cor:invertible}
    Let $(X_{1}, S_{1})$ be a cobordism from $(Y_{0}, K_{0})$ to
    $(Y_{1}, K_{1})$. Suppose there exists
    $(X_{2}, S_{2})$ such that the composite $(X,S)$ is c-equivalent
    to a product. That is, $X=X_{1}\cup X_{2}$ is diffeomorphic to a
    cylinder $[0,1]\times Y_{0}$ by a diffeomorphism $h$, and $h(S)$
    is c-equivalent to $[0,1]\times K_{0}$. Then $\zsharp(X_{1}, S_{1})$ is
    a fractional ideal, meaning there exists $A \in \cR$ such that 
            \[A\, \zsharp(X_{1}, S_{1}) \subset \cR.\]
\end{corollary}

\begin{proof}
    From the previous three lemmas, we have
\[
                \zsharp(X_{1}, S_{1}) \zsharp(X_{2}, S_{2}) \subset \cR.
\]
The submodule $\zsharp(X_{2}, S_{2})$ is non-zero (as always) so contains
some non-zero $A/B\in \Frac(\cR)$. Then the above inclusion gives 
\[
              A \,\zsharp(X_{1}, S_{1}) \subset B\, \cR.
\]
\end{proof}

\subsection*{Classical knots}

Let us focus now on the special case of classical knots in $S^{3}$,
and take $X=[0,1]\times S^{3}$. Consider connected cobordisms $S$ from
the unknot $U_{1}$ to a general knot $K$. Any two such cobordisms are
c-equivalent. Furthermore, Corollary~\ref{cor:invertible} always
applies in this situation. So we can make the following definition.

\begin{definition}\label{def:zK}
    For a classical knot $K$, we define $\zsharp(K) \subset \Frac(\cR)$ to be the
    fractional ideal $\zsharp(X,S)$, where $X=[0,1]\times S^{3}$, and $S$ is
    any connected, oriented cobordism from $U_{1}$ to $K$.
    This fractional ideal is
    independent of the choice of $S$.
\end{definition}

\begin{remarks}
In the situation described in this definition, we can construct a
cobordism $S' = D^{2}\cup S$ from the \emph{empty} link $U_{0}$ to
$K$, and we can equivalently define $\zsharp(K)$ to be $\zsharp(X,S')$. To see
that these are equal, note first that $\im\Isharp(S') \subset
\im\Isharp(S)$, so an inclusion
\[
          \zsharp(X,S) \subset \zsharp(X,S')
\]
follows from the definition. To obtain equality, note that there is
the point operator $\Lambda=\Lambda_{p}$ acting on both
$\Isharp(U_{1}; \Gamma)$ and $\Isharp(K;\Gamma)$, so  $\Isharp(S)$ is
a homomorphism of modules over the larger ring
$\cF=\cR[\Lambda]/(\Lambda^{2} + P \Lambda + Q)$. The Floer homology
$\Isharp(U_{1};\Gamma)$ is a free module of rank $1$ over $\cF$,
generated by $\bx_{+}$, and the latter element is in the image of the
map \[ \Isharp(D^{2}) : \Isharp(U_{0};\Gamma) \to
\Isharp(U_{1};\Gamma).\]
So $\zeta(X,S)$ and $\zeta(X,S')$ can both be described as the set of
$a/b$ such that
\[
               a\, \Isharp(S)(\bx_{+}) \in b \,\Isharp(K;\Gamma).
\]
\end{remarks}

From the fractional ideal $\zsharp(K)$, we can read off a constraint on the
possible genus and number of positive double points, for surfaces $S$
in $B^{4}$ which bound $K$.

\begin{proposition}\label{prop:classical-gamma-delta}
    If the classical knot $K\subset S^{3}$ bounds a surface $S$ in
    $B^{4}$ with genus $\gen$ and $\dplus$ positive double points (and any
    number of negative double points), then
    \[
             P^{\gen} \V^{\dplus}\in \zsharp(K).
     \]
\end{proposition}

\begin{proof}
    This follows from the definition and Corollary~\ref{cor:in-zXS}.
\end{proof}

As a  generalization of the above proposition, we have the following.

\begin{proposition}\label{prop:classic-cobordism}
    Let $S$ be a normally immersed cobordism from a knot $K_{0}$ to a knot
    $K_{1}$. Let $\gen$ be its genus and $\dplus$ the number of positive
    double points. Then
    \[
                P^{\gen} \V^{\dplus} \zsharp(K_{0})\subset \zsharp(K_{1}).
    \]
\end{proposition}

\begin{proof}
    Let $S_{0}$ be a cobordism from the unknot to $K_{0}$, and let
    $S_{1}$ be the composite of $S_{0}$ and $S$, from the unknot to $K_{1}$. From Lemma~\ref{lem:composite-z} we have
\[
           \zsharp(X, S_{0})\, \zsharp(X,S) \subset \zsharp(X, S_{1}),
\]
    where $X$ is $[0,1]\times S^{3}$ in each case. From
    Corollary~\ref{cor:in-zXS}, we have $P^{\gen}\V^{\dplus}\in \zsharp(X,S)$. So
    the above inclusion implies
\[
                   P^{\gen}\V^{\dplus}\zsharp(X, S_{0}) \subset \zsharp(X,S_{1}),
\]
   which is to say, $ P^{\gen}\V^{\dplus}\zsharp(K_{0}) \subset \zsharp(X,K_{1})$ as claimed.
\end{proof}

\begin{corollary}
   For a classical knot $K$,  the fractional ideal $\zsharp(K) \subset \Frac(\cR)$
   is a concordance invariant of $K$.
\end{corollary}

\begin{proof}
    We apply the previous proposition to a concordance from $K_{0}$ to
    $K_{1}$, and we see $\zsharp(K_{0}) \subset \zsharp(K_{1})$. The reverse
    inclusion holds for the same reason.
\end{proof}

We make some remarks about the concordance invariant $\zsharp(K)\subset
\Frac(\cR)$, which seems to be of interest. Previous constraints on embedded
surfaces that have been obtained using gauge theory have most often
treated genus and positive double-points even-handedly. Thus the
results of \cite{Obstruction} lead to a lower bound on $\ga(S) +
\dplus(S)$, for embedded surfaces $S$ in a fixed homology class in
a simply-connected 4-manifold. The closely-related knot invariant
$s^{\sharp}(K)$ defined in \cite{KM-s-invariant} is a cousin of
Rasmussen's $s$-invariant for knots, and has the property that
 \[\ga(S) + \dplus(S) \ge s^{\sharp}(K)/2,\] for
any oriented immersed surface in $S\subset B^{4}$ with boundary the knot. 

By contrast
with the invariant $s^{\sharp}(K)$,
the invariant $\zsharp(K)$ appears to have the potential to provide a
constraint on the pair $(\ga(S), \dplus(S))$ which is not a
constraint only on their sum. Our results say that the pair is
constrained to lie in the set
\begin{equation}\label{eq:G-set}
        G(K) =  \bigl\{ \, (\ga,\dplus) \in \mathbb{N}\times\mathbb{N} \,\bigm| \,P^{\ga}
            \V^{\dplus} \in \zsharp(K) \, \bigr\} .
\end{equation}
That said, the authors lack any resources for calculating
$\zsharp(K)$, except in some simple examples, at least at the time of
writing. By smoothing a double point, one can always decrease
$\dplus$ by one (if it is positive) in exchange for increasing
$\ga$ by one. That is, $(\ga , \dplus)$ arises as the genus
and number of double points for an immersed surface, then so does
$(\ga+1,\dplus-1)$, if $\dplus>0$. It would be interesting
to know whether the set $G(K)$ shares this property.

\section{Non-orientable surfaces}

\subsection*{Adaptation of the ideal to the non-orientable case}

As mentioned in the previous paragraph, if $\s:\cR\to \cS$ is a base change, then we can repeat the
constructions above, with $\Gamma_{\s}=\Gamma\otimes_{\s}\cS$ replacing the local system
$\Gamma$, and $\cS$-modules replacing $\cR$-modules in the
discussion throughout. We require only that $\cS$ is an integral
domain and that \emph{$\s(P)$ and $\s(\V)$ are both non-zero}.
So, to a cobordism $(X,S)$ as above, we can associate an
$\cS$-module, \[ \zsharp_{\s}(X,S)\subset \Frac(\cS), \]
in the
field of fractions $ \Frac(\cS)$ of $\cS$. (The condition that $\s(P)$
and $\s(\V)$ are non-zero is used, for example, in the proof of
Lemma~\ref{lem:c-equivalent-equal}.)
For a classical knot $K$, this provides a
fractional ideal \[\zsharp_{\s}(K)\] which is again a concordance invariant of
the knot. Proposition~\ref{prop:classical-gamma-delta} continues to
hold, and tells us that if $K$ bounds an oriented surface $S$ with
adjusted genus $\ga$ and $\dplus$ positive double points, then
\begin{equation}\label{eq:s-gamma-delta}
         \s(P)^{\ga}\s(\V)^{\dplus} \in \zsharp_{\s}(K).
\end{equation}

In this form, nothing is gained from the base change: the above
constraint on $\ga$ and $\dplus$ can only be weaker than the
previous one. However, there is a special class of cases in which
$\zsharp_{\s}(K)$ contains information also about \emph{non-orientable}
surfaces. We suppose from now on in this section that
\[
           \s(T_{0}) = 1 \in \cS.
\]
This means in particular that $\s(P)=\s(\V)$, so the constraint
\eqref{eq:s-gamma-delta} becomes
\[
             \s(P)^{\ga + \dplus} \in \zsharp_{\s}(K).
\]
Consider now a possibly non-orientable immersed surface $S$ in
$B^{4}$, with boundary the classical knot $K\subset S^{3}$. We
continue to define
$\ga(S)$ as before (Definition~\ref{def:ga}), and we call it still the adjusted genus. 
We do not have a notion of positive or negative
double point any more, so we simply write
\[
       \delta(S) = \text{ number of double points }.
\]

In addition to $\ga$ and $\delta$, there is one other numerical
invariant to record, which is the degree of the immersed normal bundle
of $S$, relative to the trivialization at $\partial S$ provided by the
0-framing of $K$. We write this as
\[
      \nu(S) = \deg NS.
\]
For an orientable surface, this can already be non-zero if $S$ has
double points but is zero if $S$ is embedded. If $S$ is
non-orientable, it may be non-zero even for an embedded surface. We
combine these and define
\begin{equation}\label{eq:eta}
        \eta(S) = \ga(S) + \frac{1}{2}\delta(S) - \frac{1}{4} \nu(S).
\end{equation}
This quantity is an integer, as will emerge below. For an orientable
surface, we have
\[
      \eta(S) = \ga(S) + \dplus(S),
\]
so the constraint \eqref{eq:s-gamma-delta} can be rewritten yet again
as
\begin{equation}\label{eq:s-gamma-delta-eta}
   \s(P)^{\eta(S)} \in \zsharp_{\s}(K). 
\end{equation}

We now have the following theorem.

\begin{theorem}\label{thm:classic-nonorientable}
      Let $\s:\cR \to \cS$ be a base-change with
     $\s(T_{0})=1$, let $K\subset S^{3}$ be a knot, and let  
    $\zsharp_{\s}(K) \subset \mathop\mathrm{Frac}(\cS)$ be the associated fractional ideal
     as
    above.
    Let $S$ be a possibly non-orientable, normally immersed connected surface in $B^{4}$,
    with boundary $K$. Then
   \[
                      \s(P)^{\eta(S)} \in \zsharp_{\s}(K),  
\]
   where $\eta(S)$ is as in \eqref{eq:eta}.
\end{theorem}

Note that in the statement of this theorem, the definition of
$\zsharp_{\s}(K)$ has not changed, and is still obtained by using an
\emph{orientable} surface, as in Definition~\ref{def:zK}.

\subsection*{Proof of the Theorem for the non-orientable case}

    The proof of the theorem follows the same basic plan as the proof
    of Proposition~\ref{prop:classical-gamma-delta} which treats the
    orientable case. The proof of that proposition arose from considering
    the effect of altering an immersed surface $S$ by finger moves,
    twist moves, and the addition of handles. For the non-orientable
    case, we consider also the effect of a connect sum with $\RP^{2}$,
    as in \cite[Lemma~\ref{1-lem:csum-R-theta}]{KM-ibn1}. 

    To carry this out, consider the immersed cobordism in $I\times S^{3}$,
    \[
         S^{*} : U_{1} \to K
    \] 
   obtained by removing a standard pair  $(B^{4}_{\epsilon}, D^{2}_{\epsilon})$
   from $(B^{4}, S)$.  
Let $\omega$ be an immersed surface in the
interior of $I\times S^{3}$ whose boundary is a
    collection of simple closed curves $\partial\omega \subset
    \mathrm{int}(S^{*})$, along which $\omega$ and $S^{*}$ meet
    cleanly. Let $\omega$ be chosen furthermore so that the curves
    $\partial\omega$ is a representative for the Poincar\'e dual of
    $w_{1}(S^{*})$ in $H_{1}(S^{*}; \Z/2)$:
\begin{equation}\label{eq:PD-w1}
            \mathrm{PD}_{S^{*}} [ \partial \omega] = w_{1}(S^{*}).
\end{equation}
 The relative homology
    class of $\omega$ is uniquely characterized by this
    condition. Corresponding to the cobordism $S^{*}$ and the surface
    $\omega$, we have a homomorphism
\[
           \Isharp(S^{*}; \Gamma_{\s})_{\omega} : \Isharp(U_{1};
           \Gamma_{\s}) \to \Isharp(K; \Gamma_{\s}).
\]
    The homomorphism is independent of the choice of $\omega$, subject
    to the constraint \eqref{eq:PD-w1}, because
    it depends only on the relative homology class. Note that the
    particular choice we made in \eqref{eq:blow-up-rule} for how to
    define $\Isharp(S^{*};\Gamma_{\s})_{\omega})$ in the case of
    \emph{immersed} rather than embedded surfaces is important at this
    point. As explained in the remarks there, the symmetry between the
    two terms on the right-hand side of \eqref{eq:blow-up-rule} is
    necessary to ensure that $\Isharp(S^{*}; \Gamma_{\s})_{\omega}$ is
    unchanged if $\omega$ is modified by an isotopy that moves
    $\partial\omega$ across one of the double-points of the
    surface.

  Let $S^{*}_{0} $ be any other immersed cobordism
   with the same boundary. As immersed surfaces in $I\times S^{3}$,
   these two differ by a sequence of operations each of which is one
   of the following or its inverse:
\begin{itemize}
\item an ambient isotopy relative to the boundary;
\item  introducing a double point by a twist  move, either positive or negative;
\item  introducing two new double points by a finger move;
\item an internal connected sum with an embedded $\RP^{2}$ of the sort
    $R_{+}$,  as in \cite[Lemma~\ref{1-lem:csum-R-theta}]{KM-ibn1}. (Recall that $R_{+}$
    is an embedded $\RP^{2}$ with self-intersection $+2$.)
\end{itemize}

\begin{remark}
It is not necessary to include a connect sum with $R_{-}$ in this
list, because the same effect can be achieved by a sum with $R_{+}$
followed by isotopies, finger moves, and twist moves. Similarly, it is
not necessary to include a sum with $T^{2}$.
\end{remark}

So let $S^{*}_{0}$, $S^{*}_{1}$, \dots, $S^{*}_{k}=S^{*}$ be a
sequence of surfaces related each to the next by one of these
operations or its inverse.
For each $S^{*}_{j}$, let $\omega_{j}$ be an immersed surface in the
interior of $I\times S^{3}$ whose boundary $\partial\omega_{j} \subset
    \mathrm{int}(S^{*}_{j})$ is dual to $w_{1}(S^{*}_{j})$
For each $j$, consider the resulting homomorphism,
\[
           \Isharp(S^{*}_{j}; \Gamma_{\s})_{\omega_{j}} : \Isharp(U_{1};
           \Gamma_{\s}) \to \Isharp(K; \Gamma_{\s}).
\]

Consider one step in this sequence: suppose that $S_{1}^{*}$ is
obtained from $S_{0}^{*}$ by one of the operations listed above. In
the case of the twist move and finger move, we can suppose that
$\omega_{0}$ is disjoint from the regions involved in the modification
of $S_{0}$, and we can take $\omega_{1}=\omega_{0}$. The situation
then is no different from the orientable case, and accordingly we have
\[
            \Isharp(S^{*}_{1}; \Gamma_{\s})_{\omega_{1}} =\U\,  \Isharp(S^{*}_{0};
                 \Gamma_{\s})_{\omega_{0}},
\]
where
\[
   \U       = \begin{cases} 
                 \s(P),&
                  \text{for the finger and positive twist moves,}\\
                 1,&
                  \text{for the negative twist move.}
              \end{cases}
\]   
In the case that $S^{*}_{1} = S^{*}_{0}\csum R_{+}$, in order to
satisfy the constraint \eqref{eq:PD-w1}, we can take
$\omega_{1}$ to be $\omega_{0} \cup \pi$, where $\pi$ is a disk
meeting $R_{+}$ in a generator of $H_{1}(R_{+})$. According to
\cite[Lemma~\ref{1-lem:csum-R-theta}]{KM-ibn1}, we then have
\[
             \Isharp(S^{*}_{1}; \Gamma_{\s})_{\omega_{1}} =  
                \Isharp(S^{*}_{0}; \Gamma_{\s})_{\omega_{0}}.
\]
At the same time, we can consider how the numerical invariant
$\eta(S)$ is changed by these operations. For the finger move,
$\nu(S)$ is unchanged, while $\delta(S)$ increases by $2$. For the
positive (respectively, negative) twist moves, $\nu(S)$ decreases
(respectively, increases) by $2$. So we have
\[
         \eta(S_{1}) = \eta(S_{0}) + \tau,
\]
where
\[
   \tau     = \begin{cases} 
                1  ,&
                  \text{for the finger and positive twist moves,}\\
                0,&
                  \text{for the negative twist move.}
                                              \end{cases}
\]
For the sum with $R_{+}$, the adjusted genus $\ga(S)$ increases by
$1/2$, and $\nu(S)$ increases by $2$, so
\[
               \eta(S_{1}) = \eta(S_{0})
\]
in this case. (Note in particular that the change in $\eta$ is always
an integer, which allows us to verify that $\eta(S) \in \Z$, as $\eta(S_{0})$
is manifestly an integer if $S_{0}$ is orientable.)

If we compare the formulae for the change in $\eta(S)$ with the
formulae for the change in $\Isharp(S;\Gamma_{\s})_{\omega}$, we
see that
\[
               \s(P)^{\eta(S_{0})} \Isharp( S_{1}^{*} ;
               \Gamma_{\s})_{\omega_{1}}
                          =  \s(P)^{\eta(S_{1})} \Isharp( S_{0}^{*} ;
               \Gamma_{\s})_{\omega_{0}}.
\]
If we apply this argument to the sequence of modifications from
$S_{0}$ to $S$, we learn
that for some $n\ge 0$,
\[
         \s(P)^{\eta(S_{0}) + n} \Isharp(S^{*}; \Gamma_{\s})_{\omega}
                 =  \s(P)^{\eta(S) + n} \Isharp(S^{*}_{0}; \Gamma_{\s})_{\omega_{0}}.
\]
From here, the proof of the theorem proceeds exactly as in the
orientable case, which is the case already established at \eqref{eq:s-gamma-delta-eta}.

\subsection*{More general non-orientable cobordisms}

As in the orientable case, the above theorem for classical knots can
be set up more generally for cobordisms of pairs $(Y, K)$. Although we
will not spell this out in full, we can usefully describe the
appropriate functorial setup. For this purpose, we need to keep track
not just of the surface $S$ in a morphism, but also the surface
$\omega$. In more detail,
the correct category has objects
$(Y_{0}, K_{0})$ and $(Y_{1}, K_{1})$, where
$Y_{i}$ is a closed oriented 3-manifold, and $K_{i}\subset Y_{i}$
is an \emph{oriented} link. Again, $B(y_{i})\subset Y_{i}$ will be a
standard ball disjoint from $K_{i}$, a neighborhood of a chosen
basepoint. 
For a morphism from from $(Y_{0}, K_{0})$ to $(Y_{1} , K_{1})$ we
require the following data:
\begin{enumerate}
\item a cobordism of pairs $(X,S)$, with $X$ an
    oriented 4-manifold and $S$ a (possibly non-orientable) immersed
    surface with transverse double-points;
\item a surface $\omega$ in the interior of $X$ whose boundary is a
    collection of simple closed curves $\partial\omega \subset
    \mathrm{int}(S)$, along which $\omega$ and $S$ meet cleanly;
\item\label{item:omega-w1} an orientation of $S\setminus\partial \omega$ which is compatible with
    the orientations $-K_{0}$ and $K_{1}$ at the boundary, and which
    changes sign across the curves $\partial\omega$;
\item an embedded cylinder $[0,1]\times B^{3}$ (or framed arc) joining
    $B(y_{0})$ to $B(y_{1})$, disjoint from $S$ and $\omega$.
\end{enumerate}

The orientation conditions imply that $\partial\omega$ represents to
the dual of $w_{1}(S)$. We will say that two
cobordisms $(X,S,\omega)$ and $(X', S', \omega')$ are isomorphic if
they are diffeomorphic relative to the boundary, respecting
orientations. Set up in this way, morphisms compose correctly. From
$\Isharp$ we obtain a functor which assigns $\Isharp(Y,K  ;
\Gamma_{\s})$ to the object $(Y,K)$, and assigns
$\Isharp(X,S;\Gamma_{\s})_{\omega}$ to the morphism $(X,S,\omega)$ as
expected.

Imitating the previous definitions used in the orientable case, we can
now formulate the following generalization of $\zsharp(X,S)$
(Definition~\ref{def:zXS-def}). Given a morphism $(X,S,\omega)$ as
just described, let $M$ be the image of
$\Isharp(X,S;\Gamma_{\s})_{\omega}$, and
let \[\zeta_{\s}(X,S,\omega)\subset \mathop\mathrm{Frac}(\cS) \]  be
the $\cS$-submodule
\[
            \zeta_{\s}(X,S,\omega) = [ \Isharp(Y_{1}, K_{1};\Gamma_{\s}) : M ].
\]
Then set
    \[
    \begin{aligned}
        \zsharp_{\s}(X,S,\omega) &= \s(P)^{\eta(S)} \zeta_{\s}(X,S,\omega) \\
        &\subset \mathop\mathrm{Frac}(\cS),
    \end{aligned}
    \]
again as in Definition~\ref{def:zXS-def}. The proof of
Theorem~\ref{thm:classic-nonorientable} adapts readily to establish
that the submodule $\zsharp_{\s}(X,S,\omega)$ is unchanged if $S$ and
$\omega$ are altered by certain standard operations. To spell this
out, let us say that $(S,\omega)$ and $(S',\omega')$ are
$\tilde{c}$-equivalent if one can be obtained from the other by a
sequence of the following moves and their inverses:
\begin{itemize}
\item an ambient isotopy relative to the boundary;
\item replacing $\omega$ with another homologous surface;
\item  altering $S$ by introducing a double point by a twist  move,
    either positive or negative, in 4-ball disjoint from $\omega$;
\item  introducing two new double points by a finger move, in a ball
    disjoint from $\omega$;
\item replacing $(S,\omega)$ by $(S\csum R_{+}, \omega\cup\pi)$, where
    $R_{+}\subset S^{4}$ is a standard $\RP^{2}$ as before, and $\pi$
    is a disk in $S^{4}$ whose
    boundary is a generating circle in $R_{+}$.
\end{itemize}
With these definitions, the statement becomes:

\begin{corollary}
    If $\s(T_{0})=1$ and $(S,\omega)$ is $\tilde c$-equivalent to\/ $(S^{*}, \omega^{*})$ then
    the corresponding modules $\zsharp_{\s}(X,S,\omega)$ and
    $\zsharp_{\s}(X, S^{*}, \omega^{*})$ are equal. \qed
\end{corollary}

As a consequence, we have a lower bound on $\eta(S^{*})$ for any
$\tilde{c}$-equivalent pair (\cf~Corollary~\ref{cor:in-zXS}):

\begin{corollary}
   If $\s(T_{0})=1$ and  $(S,\omega)$ is $\tilde c$-equivalent to $(S^{*}, \omega^{*})$, then
\[
         \s(P)^{\eta(S^{*})} \in \zsharp_{\s}(X,S,\omega).
\] \qed
\end{corollary}

\section{Reduced homology and concordance homomorphisms}

\subsection*{Using reduced homology}

Recall from \cite[section~\ref{1-subsec:reduced}]{KM-ibn1} that if $\s:\cR\to\cS$ is a base-change with
$\s(T_{0})=\s(T_{1})$, then there is a reduced variant
$\Inat(K;\Gamma_{\s})$ of the corresponding instanton homology. If we
continue to suppose that $\cS$ is at least an integral domain and
$\s(P)$ and $\s(\V)$ are non-zero, then we can use $\Inat$ in place of
$\Isharp$ to define a fractional ideal
\[
                   \znat_{\s}(K)\subset \Frac(\cS)
\]
as a variant of $\zsharp_{\s}(K)$. For the case of a knot $K$, this is
algebraically a little simpler than $\zsharp_{\s}(K)$. In this case, the
instanton homology $\Inat(K;\Gamma_{\s})$ has rank $1$. If we write
\[
    \Inat(K;\Gamma_{\s})' =  \Inat(K;\Gamma_{\s}) / (\text{Torsion}),
\]
then, being a finitely-generated, rank-1 torsion-free module over $\cS$, this quotient is
isomorphic to an ideal $\mathcal{J}_{K}$ of $\cS$ (though not uniquely). Choose such an isomorphism of $\cS$-modules,
\[
      \phi:   \Inat(K;\Gamma_{\s})' \to \mathcal{J}_{K}.
\]    
If $S$ is a cobordism of based knots, from a knot
$K_{0}$ to $K_{1}$, then we have a homomorphism
\[
           \Inat(S;\Gamma_{\s})' :  \Inat(K_{0};\Gamma_{\s})' \to  \Inat(K_{1};\Gamma_{\s})',
\]
and in the special case of a cobordism from the unknot $U_{1}$ to $K$,
\[
              \Inat(S;\Gamma_{\s})' : \cS \to  \Inat(K_{1};\Gamma_{\s})'.
\]    
Let $\iota$  be the image of $1$ under $\Inat(S;\Gamma_{\s})'$. The
reduced version of $\zeta_{\s}$ in this situation is
\[
    \begin{aligned}
        \zeta^{\natural}_{\s}(K) &= [ \Inat(S;\Gamma_{\s}) :
        \cS\iota ] \\
        &= \phi(\iota)^{-1} \mathcal{J}_{K}.
    \end{aligned}
\]
In particular, the fractional ideal  $\zeta^{\natural}_{\s}(K)$ is
isomorphic to $\Inat(K; \Gamma_{\s})'$ as a $\cS$-module. The
concordance invariant is the fractional ideal
\begin{equation}\label{eq:znat-K}
    \znat_{\s}(K) = \s(P)^{\ga} \s(\V)^{\dplus} \zeta^{\natural}_{\s}(K)
\end{equation}
(Definition~\ref{def:zXS-def}), which is therefore also isomorphic to
 $\Inat(K; \Gamma_{\s})'$ as a $\cS$-module. 

For the case of reduced homology of a knot, a somewhat more direct
definition of the ideal $\znat_{\s}(K)$ can be obtained from
the following equivalent characterization, which uses a cobordism from
$K$ to the unknot rather than the other way around. 

\begin{lemma}
    \label{lem:z-backwards}
   Let $\s:\cR\to\cS$ be a base change with $\s(T_{0})=\s(T_{1})$.
    For a classical knot $K$, let $\Sigma$ be an oriented immersed
    cobordism from $K$ to the unknot $U_{1}$. Let 
     \[
      \mathcal{I} \subset \Inat(U_{1} ; \Gamma_{\s}) \cong \cS
     \]
    be the image of $\Inat(\Sigma; \Gamma_{\s})$, regarded as an ideal
    in $\cS$ via the isomorphism. Then
     \[
             \znat_{\s}(K) = \s(P)^{-\gen(\Sigma)}\s(\V)^{-\dplus(\Sigma)} \mathcal{I},
     \]
     as fractional ideals for $\cS$, where $\gen$ and $\dplus$ are
     the genus and number of positive double points.
\end{lemma}

\begin{proof}
    Let $S$ be a cobordism from $U_{1}$ to $K$. To abbreviate our
    notation, we identify the reduced homology of $U_{1}$ with $\cS$
    and we write $N$ for the module $\Inat(K;\Gamma_{\s})/\mathrm{torsion}$. Let $i(S)$ and
    $i(\Sigma)$ denote the maps induced by these cobordisms modulo
    torsion:
    \begin{equation}
        \begin{aligned}
            i(S) &: \cS \to M \\
            i(\Sigma) &: M \to \cS. \\
        \end{aligned}
    \end{equation}
    We regard $M$ itself as a fractional ideal in $M\otimes
    \Frac{\cS}$. With that in mind,
    we have previously
    defined the fractional ideal $\zeta^{\natural}_{\s}(K)$ as
    \[
             \zeta^{\natural}_{\s}(K) = \{ \, c \in \Frac(\cS) \mid  c
             \, i(S)(1) \in M \, \}.
    \]
    The map $i(S)$ is an isomorphism from $M$ to its image $\cI\subset
    \cS$, so we can write
    \[
                       \zeta^{\natural}_{\s}(K) = \{ \, c \in \Frac(\cS) \mid  c
             \,i(\Sigma) i(S)(1) \in \cI \, \}.
    \]
     The composite cobordism $ S \cup\Sigma$ from $U_{1}$ to $U_{1}$
     gives rise to the map
     \[
                i(\Sigma) i(S)   = \sigma(P)^{\gen(S\cup\Sigma)} \sigma(L)^{\dplus(S\cup\Sigma)}.     
     \]
     So 
     \[
                   \zeta^{\natural}_{\s}(K) =  \sigma(P)^{-\gen(S\cup\Sigma)} \sigma(L)^{-\dplus(S\cup\Sigma)}\cI.
     \]     
    By definition of $\znat_{\s}$,
       \[
       \begin{aligned}
           \znat_{\s} &= \sigma(P)^{\gen(S)}
           \sigma(L)^{\dplus(S)} \zeta^{\natural}_{\s}(K) \\
                             &=  \s(P)^{-\gen(\Sigma)}\s(\V)^{-\dplus(\Sigma)} \mathcal{I}
       \end{aligned}
     \]
    as the lemma claimed. 
\end{proof}

\subsection*{Concordance homomorphisms}
\label{subsec:conc-hom}

We return to classical knots $K\subset S^{3}$ and
Proposition~\ref{prop:classic-cobordism}.
We can use this result to define homomorphisms from the knot
concordance group, in the spirit of Rasmussen's
$s$-invariant \cite{Rasmussen-slice} or the $\tau$-invariant of
Ozsvath and Szabo \cite{OS-tau}.

We consider a base change $\s : \cR \to \cS$, where $\cS$ is a
valuation ring. That is, writing $\Frac(\cS)$ for the field of fractions, we
have a surjective homomorphism of groups,
\[
              \ord: \Frac(\cS)^{\times} \to G,
          \]
where $G$ is a totally ordered group, written additively, and
\[ \cS\setminus\{0\} = \{\,a\mid \ord(a) \ge 0 \,\} ,\] (following the
conventions of \cite{AtiyahMacdonald}). Every finitely-generated
fractional ideal of $\cS$ is principal, and $\ord$ gives rise to a
bijection
\[
           \ord : \{\text{non-zero principal fractional ideals}\} \to G
       \]
with $\ord(I)\ord(J) = \ord(I) + \ord(J)$, and $\ord(I) \ge \ord(J)$ if
and only if $I\subset J$. (In this way,  the valuation group
$G$, the total order on $G$, and the homomorphism $\ord$ are all
determined up to equivalence by the structure of $\cS$ alone.) We
suppose as always that $\s(P)$ and $\s(\V)$ are non-zero,
and in this case
we have distinguished elements of the valuation group,
\[
              \begin{aligned}
                  \pi &= \ord (\s(P))\\
                  \lambda &= \ord(\s(\V)).
             \end{aligned}
\]    

Suppose now that $\s(T_{0}) = \s(T_{1})$, so that the reduced
group $\I^{\natural}(K; \Gamma_{\s})$ is defined. Consider
the fractional ideal $\znat_{\s}(K)$ associated to a knot
$K$ and the base-change $\s$, in the reduced version. Since this ideal
is finitely generated, it is principal. It is also a concordance
invariant of $K$, so we make the following definition:

\begin{definition}
    \label{def:f-type}
    Let $\s:\cR \to \cS$ be a base-change with $\s(T_{0})=\s(T_{1})$. Suppose $\cS$ is a
    valuation ring with valuation group $G$. Then we define a map
    \[
               \scrf_{\s} : \Conc \to G,
    \]
    where $\Conc$ is the knot concordance group, by $\scrf_{\s}(K) = \ord(\znat_{\s}(K))$.       
\end{definition}

\begin{proposition}
    The map $\scrf_{\s}$ is a group homomorphism.
\end{proposition}

\begin{proof}
    It is only necessary to prove that $\scrf_{\s}(K_{1} \csum K_{2})
    = \scrf_{\s}(K_{1})  + \scrf_{\s}(K_{2})$, which is equivalent to
    an equality of principal fractional ideals,
    \begin{equation}\label{eq:princ-prod}
              \znat_{\s}(K_{1}\csum K_{2}) =  \znat_{\s}(K_{1})\, \znat_{\s}(K_{2}).
     \end{equation}
    For a valuation ring such as $\cS$, just as for a principal
    ideal domain, every finitely-generated
    submodule of a finitely-generated free module is free, and every
    finitely-presented module is a direct sum of a free module and
    torsion modules of the form $\cS / A$, where $A$ is a principal
    ideal.   In particular, every finitely-presented module has a free
    resolution of length $1$ by finite-rank modules, and the K\"unneth
    theorem for a tensor product of differential modules holds in the
    same form as for principal ideal domains: there is
    a natural short exact sequence as in
    \cite[Proposition~\ref{1-prop:Kunneth-red}]{KM-ibn1}, and the sequence splits. 

    Let $\Inat(K; \Gamma_{\s})'$ denote again the quotient of
    $\Inat(K; \Gamma_{\s})$ by its torsion submodule. The fact
    that the sequence \cite[equation \eqref{1-eq:Kunneth}]{KM-ibn1} splits implies that the
    natural map
   \begin{equation}\label{eq:Kunneth-iso}
    \Inat(K_{1};\Gamma_{\s})'
                  \otimes_{\cS}
                  \Inat(K_{2};\Gamma_{\s})' \longrightarrow
                  \Inat(K_{1}\csum K_{2};\Gamma_{\s})'
   \end{equation}
   is an isomorphism of free rank-1 modules. For $i=1,2$, let
   $S_{i}$ be a based cobordism from the unknot to $K_{i}$. Let $S$ be the
   cobordism from the unknot to $K_{1}\csum K_{2}$ obtained by summing
   along the base-point arc. These three cobordisms give rise to maps
   $\phi_{1}$, $\phi_{2}$ and $\phi$ on reduced instanton homology:
   \[
   \begin{aligned}
       \phi_{i} : \cS &\to \Inat(K_{i} ; \Gamma_{\s})' \cong \cS \\
          \phi : \cS &\to  \Inat(K_{1} \csum K_{2} ; \Gamma_{\s})'  \cong \cS \otimes_{\cS} \cS \cong \cS .\\
   \end{aligned}
   \]
  The three maps are each multiplication by elements $\zeta_{1}$,
  $\zeta_{2}$ and $\zeta$, which are well-defined up to units. The
  naturality of the isomorphism \eqref{eq:Kunneth-iso} with respect to
  cobordisms implies that $\zeta \sim \zeta_{1}\zeta_{2}$.

  The definition of $\znat_{\s}$ means that
  \[
            \znat_{\s}(K_{i}) = \bigl\langle P^{\gen(S_{i})} \zeta_{i} \bigr\rangle,
   \]
   for $i=1,2$ and
  \[
            \znat_{\s}(K) = \bigl\langle P^{\gen(S)} \zeta \bigr\rangle.
  \]
 The genus is additive and $\zeta \sim \zeta_{1}\zeta_{2}$, so the
 desired equality of principal ideals \eqref{eq:princ-prod} follows.
\end{proof}

\begin{proposition}\label{prop:valuation-g-delta}
    Let $S$ be a connected, oriented, normally immersed cobordism from $K_{0}$ to $K_{1}$. Let $\gen(S)$ be its
    genus and $\dplus(S)$ the number of positive double
    points. Let $\s$ be a base change as in
    Definition~\ref{def:f-type}. Then
    \[
              \gen(S)\pi + \dplus(S)\lambda \ge    \scrf_{\s}(K_{1}) - \scrf_{\s}(K_{0}) .
    \] 
    In particular, for an oriented, immersed cobordism from the unknot $U_{1}$
    to $K$ (or equivalently an oriented immersed surface in the
    four-ball), we have
   \[
           \gen(S)\pi + \dplus(S)\lambda \ge    \scrf_{\s}(K).
   \]
    In the case of embedded surfaces, we deduce that the slice genus
    $g_{s}(K)$ satisfies
    \begin{equation}\label{eq:slice-basic}
                 g_{s}(K) \ge \frac{1}{\pi} \scrf_{\s}(K). 
    \end{equation}
\end{proposition}

\begin{proof}
    This is a consequence of Proposition~\ref{prop:classic-cobordism}
    and the definitions.
\end{proof}

If the base-change $\s$ has $\s(T_{0})=1$ in addition to having target
ring $\cS$ a valuation ring, then we can adapt the
theorem on non-orientable surfaces,
Theorem~\ref{thm:classic-nonorientable}. Parallel to the proposition
above, we then have:

\begin{proposition}\label{prop:valuation-non-orientable}
    Let $S$ be a possibly non-orientable,  connected
    surface normally immersed in 
    $B^{4}$ with boundary
    $K\subset S^{3}$. Let $\s$ be a base change as in
    Definition~\ref{def:f-type}, and suppose in addition that
    $\s(T_{0})=1$. Let $\eta(S)\in \Z$ be defined again by
    \eqref{eq:eta}. Then we have
    \[
             \eta(S)  \ge   \frac{1}{\pi} \scrf_{\s}(K) .
    \] 
\qed
\end{proposition}

In the case of an \emph{embedded} surface $S$, the inequality of the
last proposition can be written as
\[
    \frac{1}{2} \bigl( b_{1}(S) - \frac{1}{2}\nu(S)\bigr) \ge \frac{1}{\pi} \scrf_{\s}(K).
\]
As stated in the introduction, this inequality has the same form as
the inequality for non-orientable surfaces in \cite{Upsilon-2}, with
$\scrf_{\s}(K)$ replacing the invariant $\Upsilon_{K}(1)$ from
\cite{Upsilon-2} (and a different normalization). As in
\cite{Upsilon-2}, one can exploit the Gordon-Litherland inequality to
derive an inequality that does not involve the degree of the normal
bundle, $\nu(S)$, but instead involves the signature of the knot:
\[
              b_{1}(S) \ge \frac{1}{\pi}\scrf_{\sigma}(K) +
              \frac{1}{2} \mathop{\mathrm{signature}}(K).
\] 
 
Substantial lower bounds for the betti number of a non-orientable
surface bounding a given knot were first obtained by Batson
\cite{Batson}, who also observed that the torus knots $T_{2k,2k-1}$
bounds a non-orientable surfaces $S_{k}$ whose betti numbers have linear
growth in $k$. The torus knot $K=T_{2k,2k-1}$ has signature $-2k^2  + 2$, so the above
inequality implies
\[
         \frac{1}{\pi}\scrf_{\sigma}(K) \le k^{2} + O(k).
\]
On the other hand, the slice genus of this torus knot is
$(k-1)(2k-1)$, which is $2k^{2}$ to leading order. So the inequality
for the (usual orientable) slice-genus \eqref{eq:slice-basic} in these cases fails to be
sharp, by a factor of $2$ for large $k$, for base-changes with $\sigma(T_{0})=1$.

\section{Examples}
\label{subsec:Examples}

We now illustrate the workings of the concordance homomorphisms
$\scrf_{\sigma}$, for suitable base-changes $\s:\cR\to\cS$ to valuation
rings $\cS$.

\paragraph{Example A.}
Let $\bbK$ be any field extension of $\F_{2}$, and let $\Frac(\cS)$  be the
Novikov field of formal sums
\[
    \Frac(\cS) = \Bigl\{\, \sum_{\alpha\in\R} k_{\alpha} x^{\alpha}
             \Bigm | \text{ $k_{\alpha}\in \bbK$,  and $\forall C\in\R$,
               $\{\, \alpha \mid \alpha<C, k_{\alpha}\ne 0\,\}$ is finite} \,\Bigr\}
\]
The valuation group $G$ is $\R$, and for convenience we normalize the
valuation by declaring that
\begin{equation}\label{eq:quarter}
        \ord(x) = 1/4,
\end{equation}
so that
\[
    \ord \Bigl(  \sum_{\alpha\in\R} k_{\alpha} x^{\alpha}\Bigr)
    =\frac{1}{4} \min\{ \, \alpha \mid k_{\alpha}\ne 0 \,\}.
\]
The ring $\cS$ comprises as always the elements with $\ord\ge 0$
together with $0$. We take $\s : \cR \to \cS$ to have the form
\[
    \s(T_{i}) = 1 + p_{i}(x), \quad i=1,2,3,
\]
where the Novikov series $p_{i}(x)$ have $\ord(p_{i}(x)) > 0$ (and
$\s(T_{0})=\s(T_{1})$, as required for the reduced theory).
For any such choices, we have a homomorphism,
\[
         \scrf_{\s}: \Conc \to \R.
\]    

As a first case, we may take, for example,
\[
          p_{i}(x) = q_{i} x
 \]
 where $q_{1}, q_{2}, q_{3} \in \bbK$ are algebraically independent
 transcendental elements over $\F_{2}$. In this case we calculate
 \[
            \s(P) = (q_{2}^{2} q_{3}^{2} + q_{3}^{2}q_{1}^{2} +
     q_{1}^{2}q_{2}^{2}) x^{4}
            + \text{higher order in $x$},
\]     
and
\[
       \s(\V)  = (q_{1}^{4}+q_{2}^{2} q_{3}^{2} + q_{3}^{2}q_{1}^{2} +
     q_{1}^{2}q_{2}^{2}) x^{4}
       + \text{higher order in $x$}.
\]
Our convention \eqref{eq:quarter} means that both of these have order
$1$, so
$\pi=\lambda=1$, and the inequality for an immersed, oriented
cobordism $S$ in
Proposition~\ref{prop:valuation-g-delta} is
\[
    \gen(S) + \dplus(S) \ge \scrf_{\s}(K_{1}) -
    \scrf_{\s}(K_{0}).
\]

\paragraph{Example B.}
As a modification of the previous example, we may fix $r\in [0,1]$ and set
\[
            p_{1}(x) = q_{1} x^{r}
 \]
 and
 \[
              p_{2}(x) = p_{3}(x) = q_{2} x.
 \]
 In this case
 \[
     \s(P) = q_{2}^{4} q_{3}^{2} x^{4} + \text{higher order in $x$},
 \]
 and
 \[
           \s(\V) = q_{1}^{4} x^{4r} + \text{higher order in $x$}.
 \]
 So $\pi=1$ and $\lambda=r$. We write the corresponding concordance
 homomorphism as $\scrf_{r}: \Conc \to \R$. The inequality for immersed cobordisms
 becomes
 \begin{equation}\label{eq:fr}
    \gen(S) + r \dplus(S) \ge \scrf_{r}(K_{1}) -
    \scrf_{r}(K_{0}).
\end{equation}    

\paragraph{Example C.}

As a sort of limit of the previous examples,
let $\Frac(\cS)$ be the field of formal Laurent series in $x$
whose coefficients are formal Laurent series in $y$:
\[
           \Frac(\cS) = \F_{2}((y))((x)).
\]    
Let $G$ be the ordered group $\R\times \R$, lexicographically ordered
with the first entry most significant, and define the valuation, $\ord:
\Frac(\cS) \to G$, by
\[
            \ord(x^{a}y^{b}) = \frac{1}{4}(a, b).
        \]
The valuation ring $\cS$ consists of those elements of
$\F_{2}((y))[[x]]$ whose monomials $x^{a}y^{b}$ either have $a>0$ or have
$a=0$ and $b\ge 0$. Define $\s:\cR\to \cS$ by
\[
    \begin{aligned}
        \s(T_{1}) &= 1 + y, \\
        \s(T_{2}) &= 1 + x, \\
    \end{aligned}
\]
with $\s(T_{0})=\s(T_{1})$ and $\s(T_{3})=\s(T_{2})$. We calculate
\[
           \s(P) =x^{4} +  \text{higher order in $(x,y)$},
       \]
       and
       \[
             \s(\V) = y^{4} + \text{higher order in $(x,y)$}.
         \]
         So
         \[
             \begin{aligned}
                 \ord(P) &= (1,0)\\
                 \ord(\V) &= (0,1).
             \end{aligned}
         \]
This example gives rise (a priori) to a homomorphism
\[
    \scrf_{*} : \Conc \to \R\times \R
\]
with values in $(\frac{1}{4}\Z) \times (\frac{1}{4}\Z)$, lexicographically
ordered. However, since any two knots cobound an immersed annulus
(genus 0), we have a bound
\[
      \scrf_{*}(K) \le (0, N_{K})
  \]
for every knot $K$. So in fact the first component of $\scrf_{*}(K)$
is always zero and we have
\[
    \begin{aligned}
        \scrf_{*}: \Conc \to \{0\}\times \null &(\tfrac{1}{4}\Z) \\ \null
      \null  \cong \null &\tfrac{1}{4}\Z.
    \end{aligned}
\]    
This concordance homomorphism satisfies
 \begin{equation}\label{eq:fstar}
     \dplus(S) \ge \scrf_{*}(K_{1}) -
    \scrf_{*}(K_{0})
\end{equation}             
whenever $S$ is a normally immersed cobordism of genus $0$, from
$K_{0}$ to $K_{1}$. Prima facie, it says nothing about normally
immersed surfaces of
positive genus and does not bound the slice genus of a knot. As stated
in the introduction, it would be very interesting to know if there
really is a knot $K$ for which $\scrf_{*}(K)$ is larger than the slice
genus. 

\paragraph{Example \boldmath C$'$.}
One can simplify Example C while retaining its features by passing
to the ring $\cS_{1} = \F_{2}[[y]]$ by setting $x=0$. That is, we define:
\[
    \begin{aligned}
        \s_{1}(T_{0}) = \s_{1}(T_{1}) &= 1 + y, \\
         \s_{1}(T_{2})= \s_{1}(T_{3}) &= 1 . \\
    \end{aligned}
\]
We still have $\s_{1}(L)=y^{4}$ to leading order, but $\s_{1}(P)=0$. A
normally immersed cobordism $S$ of positive genus, between classical
knots, now gives the zero map on  instanton homology
groups, while immersed cobordisms of genus $0$ give homomorphisms
of rank $1$, 
between modules of rank $1$ over $\cS_{1}$. The concordance
homomorphism $\scrf_{*}$ in this example satisfies the same inequality
\eqref{eq:fstar} as above.
The set-up here is very
close to that of \cite{KM-s-invariant}, though we are now working in
characteristic $2$ rather than characteristic $0$. In
\cite{KM-s-invariant}, the counterpart of $\sigma_{1}(P)$ was non-zero
but divisible by $2$.

It is interesting to note that, since $\s_{1}(P)=\s_{1}(Q)=0$ in this
example, the resulting spectral sequence has $E_{2}$ page the undeformed
reduced Khovanov homology (tensored by $\cS_{1}$), so we have 
\[
     \tilde{\kh}(\bar{K}) \otimes \F_{2}((y)) \implies \Inat(K ; \Gamma_{\sigma_{1}}).
\]    

\paragraph{Example D.}  If the base-change $\s$ satisfies
$\s(T_{0})=\s(T_{1})=1$, then the resulting concordance homomorphism
$\scrf_{\s} : \Conc \to G$ provides a bound for the topology of
non-orientable immersed cobordisms, as a consequence of
Theorem~\ref{thm:classic-nonorientable} and
Proposition~\ref{prop:valuation-non-orientable}. As a particular case,
let $\cS$ be the ring of formal power series $\F_{2}[[x]]$, and define
$\s$ by
\[
\begin{aligned}
    \s(T_{0}) = \s(T_{1}) = 1 \\
           \s(T_{2}) = \s(T_{3}) = 1+x.
\end{aligned}
\]
In this example $P$ and $\V$ are equal, and $\pi=\lambda=1$ if we set
$\ord(x)=1/4$ as before. This gives rise to a
concordance homomorphism $\scrf_{\s}:\Conc\to\Z$ with the property
that if $S$ is a possibly non-orientable cobordism from $K_{0}$ to
$K_{1}$, then
\[
            \eta(S) \ge \scrf_{\s}(K_{1}) -
    \scrf_{\s}(K_{0}),
\]
where $\eta$ is as in equation \eqref{eq:eta}.

\paragraph{Example E.} The
following is a hypothetical example, to illustrate the potential
workings of the concordance invariants $\scrf_{\s}$. Let
$\s  :\cR \to \cS$ be the quotient map by the ideal
$\langle T_{1}-T_{0}\rangle$ in $\cR$. Thus $\cS$ is the ring
$\cS_{\BN}$ of the introduction, a ring of Laurent series in $T_{1}$,
$T_{2}$, $T_{3}$. This is the
largest quotient for which the reduced homology $\Inat(K
;\Gamma_{\s} )$ is defined. We write the images of $P$ and $\V$ in
$\cS$ simply as $P$ and $\V$ again. Suppose that the chain
complex that computes $\Inat(K
;\Gamma_{\s} )$ is chain-homotopy equivalent to
\[
    \begin{gathered}
        \cS = C_{0}\stackrel{\partial} {\longrightarrow}
        C_{1}= \cS \oplus \cS,\\
               1 \mapsto (\V^{3}, P).
    \end{gathered}
\]    
The homology is then generated by $2$ elements $[u], [v]$, with chain
representatives in $C_{1}$, satisfying \[\V^{3} [u] = P\, [v].\] Suppose
there is a genus-1 embedded cobordism $S$ from the unknot to $K$, and
that the resulting map of reduced homologies on the chain level is
\[
   \begin{gathered}
        \cS\longrightarrow
        C_{1} \\
               1 \mapsto u.
           \end{gathered}
\]           
The corresponding fractional ideal $\zeta^{\natural}_{\s}(S)$ is then
generated by $1$ and $P^{-1}\V^{3}$. The concordance-invariant ideal
$\znat_{\s}(K)$ is $\langle P, \V^{3}\rangle$. In this hypothetical
case, Proposition~\ref{prop:classical-gamma-delta} would allow that
$K$ bounds an immersed disk in the 4-ball with three positive double
points, but no fewer. The concordance homomorphism
$\scrf_{r}:\Conc\to\R$ from \eqref{eq:fr} can then be calculated using
the universal coefficient theorem. It will take the values,
\[
           \scrf_{r}(K) =
           \begin{cases}
              3r, & 0< r \le 1/3, \\
              1, &1/3\le r < 1.
           \end{cases}
       \]
The constraint \eqref{eq:fr}  on the genus and the number of positive
 double-points coming from $\scrf_{r}$ will therefore be
 \[
               \gen + r\dplus \ge           \begin{cases}
               3r, & 0\le r \le 1/3 ,\\
               1, &1/3\le r \le 1.
           \end{cases}
 \] 
Taken together over all $r$ in $[0,1]$, these constraints are equivalent to
the constraint coming from $\znat_{\s}$. That is, if $S$ is embedded then it
must have genus $1$ at least, while if $S$ is immersed with genus $0$
then it must have at least three positive double points. The
concordance homomorphism $\scrf_{*}(K)$ in such an example would be $3$.

Of course,
the exponent $3$ in this hypothetical example is arbitrary. We shall
show in section~\ref{subsec:trefoil} that the positive trefoil behaves in this way, but with the
less interesting exponent $1$ in place of $3$. (The positive trefoil
bounds both an embedded surface of genus $1$ and a disk with one
positive double point.)

\section{Unknotting number and other properties}

\subsection*{Unknotting number}

As in \cite{Alishahi,Alishahi-Dowlin,Alishahi-Eftekhary}, one can exploit the torsion in $\Isharp(K;
\Gamma_{\s})$ or $\Inat(K;\Gamma_{\s})$ instead of the torsion-free
quotient, to obtain bounds on the unknotting number of $K$, or the crossing-change
distance between knots.  Suppose $K_{0}$ can be obtained from $K_{1}$
by $n$ crossing-changes, and let $S_{10}$ be the corresponding
immersed cylindrical cobordism from $K_{1}$ to $K_{0}$. The composite
cobordism $S$ from $K_{1}$ to $K_{1}$, formed as the union of $S_{10}$
and its mirror image, has $2n$ double points which come in mirror
pairs. This cobordism can be obtained from the trivial cylinder by
doing $n$ finger-moves, and intermediate isotopies. It follows that
\[
       \Isharp(S; \Gamma) = \V^{n}.
\]    
In particular then, the map
\begin{equation}\label{eq:factors}
          \V^{n} : \Isharp(K_{1}; \Gamma) \to \Isharp(K_{1} ;\Gamma)
\end{equation}
factors through $\Isharp(K_{0};\Gamma)$. To draw a concrete
consequence from this, note that $\Isharp(K_{0};\Gamma)$ and
$\Isharp(K_{1};\Gamma)$ both have rank $2$, from which it follows that
any torsion element $w$ in the image of any full-rank homomorphism
$\phi : \Isharp(K_{0};\Gamma) \to \Isharp(K_{1};\Gamma)$ is
necessarily of the form $\phi(v)$ for some torsion element in
$\Isharp(K_{0};\Gamma)$. So, in \eqref{eq:factors}, the restriction of
the multiplication map by $\V^{n}$ to the torsion part of
$\Isharp(K_{1}; \Gamma)$ factors through the torsion part of
$\Isharp(K_{0}; \Gamma)$.

We can therefore deduce:

\begin{proposition}[\cf~\cite{Alishahi,Alishahi-Dowlin,Alishahi-Eftekhary}]
\label{prop:unknotting}
    If $K_{0}$ can be obtained from $K_{1}$ by $n$ crossing changes,
    and if\/ $H \in \cR$ annihilates the torsion in
    $\Isharp(K_{0};\Gamma)$, then $\V^{n}H$ annihilates the torsion in
    $\Isharp(K_{1};\Gamma)$.
\end{proposition}

As a special case, taking $K_{0}$ to be the unknot:

\begin{corollary}[\cf~Theorem 1.2 of \cite{Alishahi}]
\label{cor:unknotting}
    If $K$ has unknotting number $n$, then $\V^{n}$ annihilates the
    torsion in $\Isharp(K;\Gamma)$. In particular, for any knot $K$,
    the torsion in $\Isharp(K;\Gamma)$ is annihilated by some power of $\V$.
\end{corollary}

The result can be recast if we apply a base-change $\s:
\cR\to \cS$ to a valuation ring $\cS$. Let the torsion submodule of
$\Isharp(K_{1};\Gamma_{\s})$  be 
\[
         (\cS / I_{1}) \oplus (\cS / I_{2}) \oplus \cdots \oplus (\cS/I_{l})
     \]
with $\ord(I_{1}) \ge \ord(I_{2}) \ge \dots \ge \ord(I_{l})$ in
$\mathrm{Val}(\cS)$. Thus $I_{1}$ is the annihilator of the torsion part.
Let $\tau(K_{1})=\ord(I_{1})$. Define $\tau(K_{0})$
similarly for the other knot. Then the above factorization requires (as a special case),
\[
               n\lambda \ge \tau(K_{1}) - \tau(K_{0}),
           \]
where $\lambda$ is the order of $\s(\V)$ as before.           
This goes both ways, so the crossing-change distance between the two knots is therefore
bounded below by
\[
     \min\bigl\{ \, n  : n\lambda \ge | \tau(K_{1}) - \tau(K_{0}) | \, \bigr\},
\]    
or simply by \[ \frac{| \tau(K_{1}) - \tau(K_{0}) | }{ \lambda}\] if
$\mathrm{Val}(\cS)\subset \R$. As a special case, with $K_{0}$ the
unknot again, the unknotting number of $K$ is bounded below by
$\tau(K)/\lambda$, where $\tau(K)$ is the order of the annihilator of
the torsion submodule in $\Isharp(K;\Gamma_{\s})$.

    There is a slight strengthening of
    Proposition~\ref{prop:unknotting} and
    Corollary~\ref{cor:unknotting}. We are free to modify the
    definition we have made for how to define the map
    $\Isharp(S;\Gamma_{\sigma})$ when $S$ has double points, by using
    the modified formula at \eqref{eq:blow-up-rule-xi}, so defining a
    modified functor $\Isharp_{\xi}$. Recall that $\Isharp(K;\Gamma)$
    is unchanged, as only the maps induced by immersed surfaces are
    modified.
    In this case,
    the formula for $L$ must be replaced by
 \[ \V_{\xi}  =  \xi P + T_{0}^{2} +
     T_{0}^{-2} .\]
     (See \cite[Proposition~\ref{1-prop:twist-and-finger-formulae-xi}]{KM-ibn1}.)
     The conclusions of Proposition~\ref{prop:unknotting} and
      Corollary~\ref{cor:unknotting} continue to hold with $L_{\xi}$
      in place of $L$. That is, for example, if $K$ has unknotting
      number $n$, then $L_{\xi}^{n}$ annihilates the torsion in
      $\Isharp(K;\Gamma)$. Here $\xi$ is arbitrary, so we can take our
      ground ring to be $\cR[\xi]$ where $\xi$ is an
      indeterminate. The statement that $L_{\xi}^{n}$ annihilates the
      torsion is then equivalent to the statement that $L^{a}P^{b}$
      annihilates the torsion, for all $a$, $b$ with $a+b=n$. We
      record this as a variant of Corollary~\ref{cor:unknotting}.

\begin{corollary}
\label{cor:unknotting'}
    If $K$ has unknotting number $n$, then the
    torsion in $\Isharp(K;\Gamma)$ is annihilated by the ideal
    $\langle L, P \rangle^{n}$.
\end{corollary}
      
If we pass to a valuation ring, then $\ord(L)$ will in general be less
than or equal to $\ord(P)$, in which case the ideal $\langle L ,
P\rangle$ is simply $\langle L \rangle$, so the variant is equivalent
to the original in this case.

\subsection*{Ribbon concordance}

Related to the above arguments involving the unknotting number, the
functoriality of both Khovanov homology and Heegaard knot Floer
homology has been used by Zemke \cite{Zemke} and Levine-Zemke
\cite{Levine-Zemke}, to obtain constraints on the existence of a
ribbon concordance from a knot $K_{0}$ to a knot $K_{1}$. (A ribbon
concordance is an embedded annulus in $[0,1]\times S^{3}$ such that
the first coordinate function is Morse and has no index-2 critical
points.) Since the argument is quite formal, it adapts to the case of
instanton homology without essential change:

\begin{theorem}[\cf~\cite{Zemke, Levine-Zemke}]
    If $S$ is a ribbon concordance from $K_{0}$ to $K_{1}$, then the
    resulting map $\Isharp(S; \Gamma)$ is injective.   
\end{theorem}

\noindent
As in \cite{Zemke}, the proof proceeds by considering the composite of
$S$ with its mirror image, which is a cobordism from $K_{0}$ to
$K_{0}$ which can be repeatedly be simplified by neck-cutting.

\section{Calculation for the trefoil}
\label{subsec:trefoil}

\subsection*{Statement of the result}

In this section we take $K$ to be the trefoil. We work again with the
reduced homology $\Inat(K; \Gamma_{\s})$, and we take the base change
$\s: \cR\to \cS$ to the largest quotient of $\cR$ for which the
reduced theory is defined. As in the introduction and Example E from
Section \ref{subsec:conc-hom}, this ring $\cS$ is the Laurent series
ring $\cS_{\BN}$ in three variables $T_{i}$, $i=1,2,3$, and $\s_{\bn}$ is
the quotient map by the ideal generated by $T_{0}-T_{1}$. The local
system $\Gamma_{\s_{\bn}}$ coincides with the system named $\Gamma_{\BN}$ in
the introduction. We again simply write $P$ and $L$ for the Laurent polynomials
which are really $\s_{\bn}(P)$ and $\s_{\bn}(L)$. In particular then,
\[
           L = T_{1}^{2} + T_{1}^{-2} + P.
\]
We write $\znat_{\BN}(K)$ for  $\znat_{\s_{\bn}}(K)$ in this
case, and we will
compute the fractional ideal
$\znat_{\BN}(K) \subset \Frac(\cS_{\BN})$ for the trefoil.

\begin{proposition}\label{prop:trefoil}
    For the right-handed trefoil $K_{2,3}$, the complex of free $\cS_{\BN}$-modules
    that computes $\Inat(K_{2,3}; \Gamma_{\BN})$ is chain-homotopy equivalent
    to the complex
    \[
              \cS_{\BN} \stackrel{\partial}{\to} \cS_{\BN} \oplus \cS_{\BN}
    \]
    where $\partial(1) = (L, P)$. In particular,  $\Inat(K_{2,3};
    \Gamma_{\BN})$ has  a presentation with two generators and one
    relation, $Le_{1} + Pe_{2}=0$, which means that it isomorphic to
    the ideal $J= \langle P, L \rangle$ as an
    $\cS_{\BN}$-module. Furthermore, the
   fractional ideal $\znat_{\BN}(K_{2,3})$ coincides with $J$.
\end{proposition}

\begin{remark}
     After the preparation of earlier drafts of this paper, it became
     apparent that it is possible to prove this result with no reference to
    ``instantons'' beyond what is already built into the
    formal properties of $\Inat(K;\Gamma)$. 
    Nevertheless, in the proof we give here (in particular
    in Lemma~\ref{lem:Xpm} below), we obtain some explicit information
    at the chain level by considering instanton moduli spaces, so
    making contact with the constructions that underlie the
    definitions. The authors therefore decided to retain this version
    of the calculation.
\end{remark}

\subsection*{The skein triangle and the Hopf link}

    The first step in the proof of Proposition~\ref{prop:trefoil}
    is to use the skein exact triangle
illustrated in Figure~\ref{fig:skein-trefoil}. The figure shows the
right-handed trefoil together with the result of smoothing one of the
three crossings in two different ways. Because we are using reduced
instanton homology, we require a base-point on each link, where the
bigon is introduced. The location of the base-point is marked by a
dot in the figure.

\begin{figure}
    \begin{center}
        \includegraphics[scale=0.40]{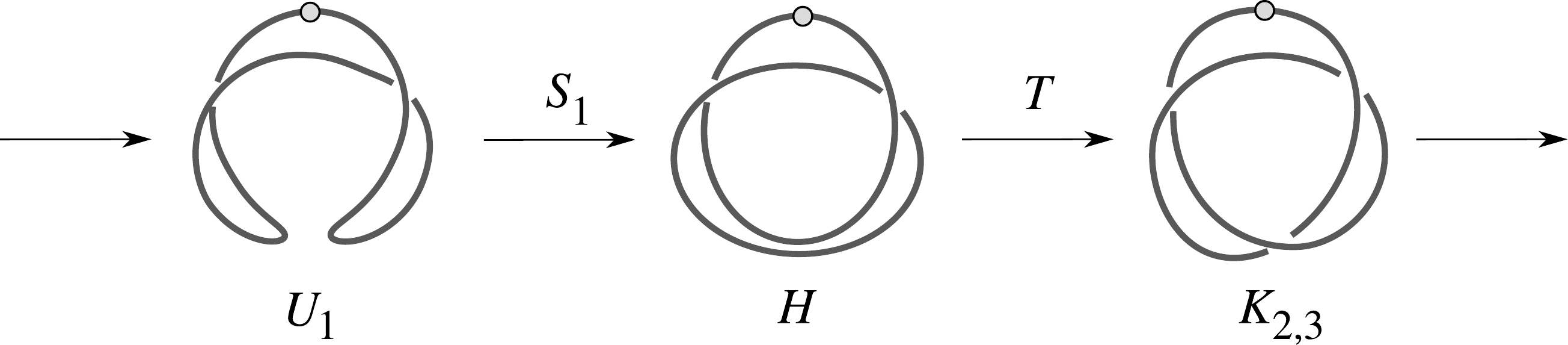}
    \end{center}
    \caption{\label{fig:skein-trefoil}
    The skein triangle for the positive trefoil $K_{2,3}$, the Hopf
    link $H$, and the unknot $U_{1}$.}
\end{figure}

The skein sequence
\begin{equation}\label{eq:skein-U-H-K}
         \cdots\to   U_{1} \to H \to K_{2,3} \to  U_{1} \to \cdots
\end{equation}
leads to a long exact sequence of instanton homology groups. Because
$\Inat(U_{1}; \cS_{\BN})$ is free of rank $1$ and the instanton
homologies of $H$ and $K_{2,3}$ have rank $2$ and $1$ respectively,
the long exact sequence must break into a short exact sequence,
\begin{equation}\label{eq:skein-trefoil-1}
 0\to \Inat(U_{1}; \cS_{\BN}) \stackrel{n}{\to} \Inat(H; \cS_{\BN})\stackrel{k}{\to}
 \Inat(K_{2,3}; \cS_{\BN}) \to 0.
\end{equation}
At the chain level, the skein sequence tells us that the corresponding
complexes $C^{\natural}(U_{1})$, $C^{\natural}(H)$ and
$C^{\natural}(K_{2,3})$ are related in such a way that
$C^{\natural}(K_{2,3})$ is chain-homotopy equivalent to the mapping
cone of the chain map arising from the cobordism
\begin{equation}\label{eq:mapping-cone}
               U_{1} \stackrel{S_{1}}{\to} H.
\end{equation}
This cobordism $S_{1}\subset [0,1]\times S^{3}$ is a pair of co-pants, but
not with the standard embedding. To calculate the complex for the
trefoil, up to chain homotopy, we shall calculate the map arising from
the cobordism $S_{1}$ at the chain level.

Before proceeding, we note for later use that we may consider the
skein triangle obtained from the smoothings of a crossing on the Hopf
link to obtain the sequence in Figure~\ref{fig:skein-hopf}. Similar
consideration of ranks shows that this gives rise to a short exact
sequence,
\begin{equation}\label{eq:skein-hopf-1}
 0\to \Inat(U_{1}; \cS_{\BN}) \stackrel{i}{\to} \Inat(H; \cS_{\BN}) \stackrel{p}{\to}
 \Inat(U_{1}; \cS_{\BN}) \to 0.
\end{equation}

\begin{figure}
    \begin{center}
        \includegraphics[scale=0.40]{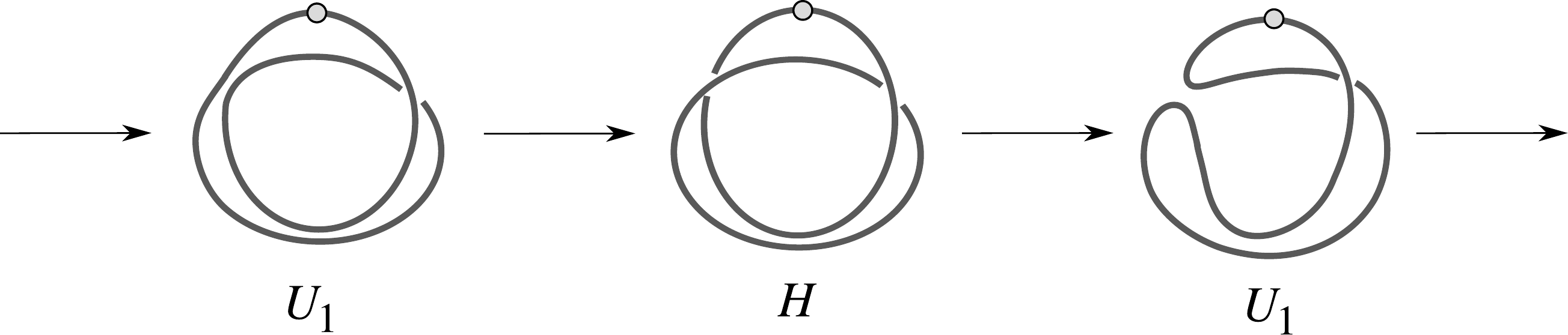}
    \end{center}
    \caption{\label{fig:skein-hopf}
    The skein triangle for the Hopf
    link $H$, and two copies of the unknot $U_{1}$.}
\end{figure}

Let us write $R^{\natural}(K)$ for the representation variety of
marked bifold connections on $K^{\natural}$, where $K^{\natural}$ is
obtained from $K$ by adding the bigon
\cite[Figure~\ref{1-fig:reduced}]{KM-ibn1}. After orienting $K$ near the base-point,
let $\m_{0}$ be any representative of the oriented meridian at the
base-point, as an element of $\pi_{1}(S^{3}\setminus K)$. Let $\bi$ be
the element $\diag(-i,i)$ in $\SU(2)$.
We can identify $R^{\natural}(K)$ with the space of representations of the
link complement,
\[
        \rho: \pi_{1}(S^{3}\setminus K) \to \SU(2)
\]
satisfying the constraint that $\rho(\m_{0})=\bi$ and $\rho(\m)$ is
conjugate to $\bi$ for all other meridians. A representation
$\rho$ gives rise to representation of the orbifold fundamental group
of the web $K^{\natural}$ by sending the meridians of the edges
$e_{1}$ and $e_{2}$ in \cite[Figure~\ref{1-fig:reduced}]{KM-ibn1} to $\bj$ and $\bk$.

The representation variety $R^{\natural}(U_{1})$ for the unknot, with
this description, consists of a single representation $\alpha$, with
$\alpha(\m_{0})=\bi$. The representation variety of the Hopf link $H$
consists of two representations: the fundamental group of the
complement is abelian, so a representation that maps $\m_{0}$ to
$\bi$ maps a meridian of the other component to $\pm \bi$. To
distinguish consistently between the two cases, given $\beta\in
R^{\natural}(H)$, we can orient the two components of the link so that
the both oriented meridians map to $\bi$. Oriented in this way, the
linking number of the Hopf link will be either $1$ or $-1$. We name
the two elements of $R^{\natural}(H)$ as $\beta_{+}$ and $\beta_{-}$
respectively. See Figure~\ref{fig:U1-Hopf}.

\begin{figure}
    \begin{center}
        \includegraphics[scale=0.60]{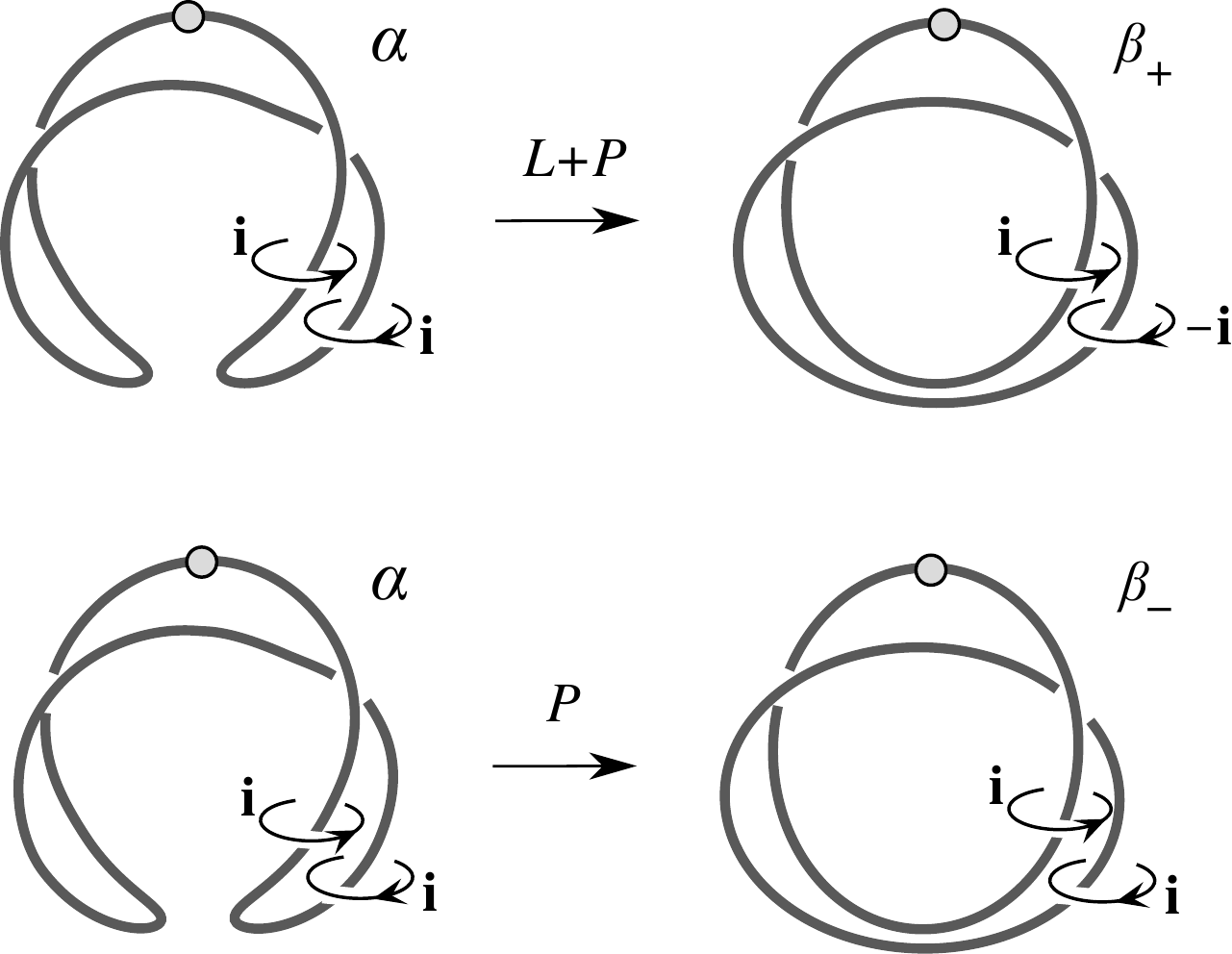}
    \end{center}
    \caption{\label{fig:U1-Hopf}
    The matrix entries $X_{+}=L+P$ and $X_{-}=P$, from $\alpha$ to
  $\beta_{+}$ and $\beta_{-}$ respectively. All loops are based, with
  the basepoint lying above the plane of the diagram as usual. Not
  shown are the two arcs of the added bigon, where the monodromies are
$\mathbf{j}$ and $\mathbf{k}$.}
\end{figure}

The critical points $\alpha$, $\beta_{+}$ and $\beta_{-}$ can all be
seen to be regular. So the corresponding chain complexes are $\cS_{\BN}$ and
$\cS_{\BN}\oplus\cS_{\BN}$ respectively. Furthermore, there is no differential in the latter
case. One can see this either by showing that $\beta_{\pm}$ have the same mod $2$ grading in this
complex, or by noting that the matrix entry of the differential from
$\beta_{+}$ to $\beta_{-}$ is equal to the matrix entry from
$\beta_{-}$ to $\beta_{+}$ by symmetry, and noting that a non-zero
entry would contradict $d^2=0$. 

The mapping cone arising from \eqref{eq:mapping-cone}
therefore has the form
\[
             X = (X_{+}, X_{-}) : \cS_{\BN} \to \cS_{\BN}  \oplus \cS_{\BN}
\]
where the two elements $X_{+}, X_{-}\in \cS_{\BN}$ are the matrix
entries at the chain level of the map
induced by the cobordism $S_{1}$, from $\alpha$ to $\beta_{+}$ and
$\beta_{-}$. (At the level of homology, this is the map $n$ in \eqref{eq:skein-trefoil-1}.) These matrix entries are determined in the next lemma,
illustrated in Figure~\ref{fig:U1-Hopf}.

\begin{lemma}\label{lem:Xpm}
    The elements $X_{+}$ and $X_{-}$ are $L+P$ and $P$ respectively.
\end{lemma}

Apart from its last sentence (identifying the fractional ideal),
Proposition~\ref{prop:trefoil} follows from the lemma. Only a change
of basis is needed to change the matrix entries in the presentation
from $(L+P,P)$ to $(L,P)$.

\subsection*{Proof of the lemma: computing $X_{+}$ and $X_{-}$}

We turn to the proof of the lemma.
We introduce two additional cobordisms, both from $H$ to $U_{1}$, called
$S_{g}$ and $S_{\delta}$ respectively. The cobordism $S_{g}$ is the
mirror image of $S_{1}$, so is an embedded pair of pants. The
cobordism $S_{\delta}$ will be an immersed cobordism with a single
double point: it is the union of an embedded cylinder $[0,1]\times
S^{1}$ where the $S^{1}$ is the component of $H$ with the base-point,
and an embedded disk $D$ whose boundary is the other component of
$H$. The disk $D$ meets the cylinder in one point. The composites
$S_{1}\circ S_{g}$ and $S_{1}\circ S_{\delta}$ are
two cobordisms from $U_{1}$ to $U_{1}$. They are respectively an
embedded surface of genus $1$, and an embedded cylinder with one
positive double point. We therefore have
\begin{equation}\label{eq:PL}
\begin{aligned}
    \Inat(S_{1}\circ S_{g}; \Gamma_{\BN}) &= P \\
    \Inat(S_{1}\circ S_{\delta}; \Gamma_{\BN}) &= L .
\end{aligned}
\end{equation}
By examining the flat connections explicitly, we will see that the
cobordism $S_{g}$ maps the generators $\beta_{\pm}$ as follows,
\begin{equation}\label{eq:Sg}
\begin{aligned}
    \Inat(S_{g}; \Gamma_{\BN})(\beta_{+}) &= 0 \\
    \Inat(S_{g}; \Gamma_{\BN})(\beta_{-}) &= \alpha \\
\end{aligned}
\end{equation}
while for $S_{\delta}$ we have
\begin{equation}\label{eq:Sdelta}
\begin{aligned}
    \Inat(S_{\delta}; \Gamma_{\BN})(\beta_{+}) &= \alpha \\
    \Inat(S_{\delta}; \Gamma_{\BN})(\beta_{-}) &= \alpha .\\
\end{aligned}
\end{equation}
From the formulae \eqref{eq:Sg} and \eqref{eq:Sdelta} and
\eqref{eq:PL}, we obtain
\[
\begin{aligned}
    \Inat(S_{1}; \Gamma_{\BN})(\alpha) &= (L+P)\beta_{+} + P \beta_{-}
\end{aligned}
\]
which is equivalent to the statement of the lemma.

To complete the proof of the lemma, it remains to prove the formulae \eqref{eq:Sg} and \eqref{eq:Sdelta}
for the cobordisms $S_{g}$ and $S_{\delta}$. Let $H^{\natural}$ and
$U_{1}^{\natural}$ be the webs obtained from $H$ and $U_{1}$ by adding
bigons near the marked point on each. Let $\hat S_{\delta}$ be the
proper transform of $S_{\delta}$ after blowing up up at the point of
self-intersection. This surface is the disjoint union of an embedded
annulus and a disk $D_{\delta}$ whose boundary is the unmarked
component of $H$. Let $S_{g}^{\natural}$ and
$\hat S_{\delta}^{\natural}$ denote the foams obtained  from
$ S_{g}$ and $\hat S_{\delta}$ by adding a bigon along arcs joining the
marked points. These foams are cobordisms from
$H^{\natural}$ to $U_{1}^{\natural}$. Let us write
\[
\begin{aligned}
    W_{g} &= ([0,1]\times S^{3}, S^{\natural}_{g}) \\
    W_{\delta} &=  ([0,1]\times S^{3} \csum \bar{\CP}^{2}, \hat S^{\natural}_{\delta}) \\
\end{aligned}
\]
for the corresponding 4-dimensional bifold cobordisms. The matrix
entries of $I^{\natural}(S_{g};\Gamma_{\BN})$ are defined by counting
instantons in zero-dimensional components of the moduli spaces
\[
                 M( \beta_{\pm} ; W_{g} ; \alpha)
\]
on the cobordism $W_{g}$ with cylindrical ends. In the case of the
double-point cobordism, the matrix entries of
$I^{\natural}(S_{\delta};\Gamma_{\BN})$ are defined by a similar count in
moduli spaces
\[
           M( \beta_{\pm} ; W_{\delta} ; \alpha)
           \quad\text{and}\quad
         M( \beta_{\pm} ; W_{\delta} ; \alpha)_{\epsilon}
\]
where $\epsilon$ is the exceptional set of the blow-up. (See the
definition at \cite[equation \eqref{1-eq:blow-up-rule}]{KM-ibn1}.)

There are smooth Klein-four-group
covers of these bifolds,
\[
\tilde W_{g} \to W_{g}\quad \text{and} \quad \tilde W_{\delta} \to W_{\delta}, 
\]
branched over the singular loci $S_{g}^{\natural}$ and $\hat S_{\delta}^{\natural}$
respectively. The trivial $\SO(3)$ bundles on $\tilde W_{g}$ and
$\tilde W_{\delta}$
descend to flat $\SO(3)$ bifold connections $c_{g}$ on $W_{g}$ and
$c_{\delta}$ on
$W_{\delta}$ respectively. The flat $\SO(3)$ bifold connection $c_{g}$
lifts to a unique (flat) $\SU(2)$ connection $C_{g}$ on $W_{g}$ and
defines an element of the moduli space $ M( \beta_{-} ; W_{g} ;
\alpha)$. Similarly $c_{\delta}$ lifts to a flat $\SU(2)$ connection
$C_{\delta}^{-}$, which defines an element of the moduli space
$M(\beta_{-} ; W_{\delta} ; \alpha)$. On the bifold $W_{\delta}$ however there is a flat
line bundle $\xi$ with holonomy $-1$ on the links of both the exceptional
sphere $\epsilon$ and the disk $D_{\delta}\subset \hat S_{\delta}$. By
twisting $C_{\delta}^{-}$ with $\xi$ we obtain an $\SU(2)$
connection $C_{\dplus}$ in $M(\beta_{+}; W_{\delta};
\alpha)$. Altogether we have three flat connections,
\begin{equation}\label{eq:three}
\begin{aligned}[]
    [C_{g}] &\in M(\beta_{-} ; W_{g} ; \alpha) \\
    [C_{\delta}^{-}] &\in M(\beta_{-} ; W_{\delta} ; \alpha) \\
    [C_{\delta}^{+}] &\in M(\beta_{+} ; W_{g} ; \alpha)_{\epsilon}.
\end{aligned}
\end{equation}
The Klein-four-group covers $\tilde W_{g}$ and $\tilde W_{\delta}$ are
two cobordisms from the rational homology sphere $\RP^{3}\csum
\RP^{3}$ to $S^{3}$, and both
have $b_{1}=0$ and $b_{2}^{+}=0$. It follows that these three elements \eqref{eq:three}
have no infinitesimal deformations and are regular points of their
respective moduli spaces. Because they are flat, the curvature
integrals \cite[equation \eqref{1-eq:nu-stokes}]{KM-ibn1} defining the local systems are trivial,
and each of the three connections therefore contributes $1$ to the
corresponding matrix entry of the map $I^{\natural}(S_{g};\Gamma_{\BN})$
or $I^{\natural}(S_{\delta};\Gamma_{\BN})$. There are no other flat
connections, and any non-flat connection would belong to a moduli
space of strictly positive dimension and would not contribute to the
cobordism maps. This completes the verification of the formulae
\eqref{eq:Sg} and \eqref{eq:Sdelta} and so completes the proof of Lemma~\ref{lem:Xpm}.

\subsection*{Identifying the fractional ideal}

We have now completed the proof of the assertion in
Proposition~\ref{prop:trefoil} that $\Inat(K_{2,3};\Gamma_{s})$ is
isomorphic to the ideal $J = \langle P, L\rangle$. To identify the
fractional ideal $\znat_{\BN}(K_{2,3})$ we need an oriented, immersed
cobordism $\Sigma$ from $U_{1}$ to $K_{2,3}$, and for this we can take the
composite of the cobordism $T$ from $H$ to $K_{2,3}$ in
Figure~\ref{fig:skein-trefoil} and the immersed cobordism $S^{\dag}_{\delta}$ from $U_{1}$
to $H$ obtained by reversing the cobordism $S_{\delta}$ from the
lemma. Proposition~\ref{prop:trefoil} identifies the
$\Inat(K_{2,3};\Gamma_{\BN})$ in terms of generators $e_{1}$ and
$e_{2}$ with a relation $Le_{1} + Pe_{2}=0$, but the lemma identifies
the same group in terms of generators $[\beta_{+}]$ and $[\beta_{-}]$
with the relation $(L+P)[\beta_{+}] + P [\beta_{-}]=0$. The change of
basis between these two descriptions is 
\[
               e_{1} = [\beta_{+}]  , \qquad  e_{2} = [\beta_{+}] + [\beta_{-}].
\]
Our discussion of $S_{\delta}$ in the proof of the lemma adapts
readily to similar case of $S_{\delta}^{\dag}$ and shows that this
cobordism from $U_{1}$ to $H$ gives the map on generators
\[
            \alpha \mapsto \beta_{+} + \beta_{-}.
\]
The composite cobordism $\Sigma$ from $U_{1}$ to $K_{2,3}$ is
therefore
\[
       \alpha \mapsto e_{2}.
\]
Since $Pe_{2} = L e_{1}$ the definition of $\zeta^{\natural}_{\BN}(\Sigma)$
shows that this fractional ideal is generated by $1$ and
$L^{-1}P$. Accordingly, from the definition \eqref{eq:znat-K}, the
ideal $\znat_{\BN}(K_{2,3})$ is $\langle L, P \rangle$ as
Proposition~\ref{prop:trefoil} claimed.  This completes the proof of
the proposition.  

\subsection*{Left-handed trefoils and the concordance homomorphisms}

The complex that computes the homology of the left-handed trefoil is
the dual complex:

\begin{proposition}\label{prop:trefoil-minus}
    For the left-handed trefoil $K_{2,3}^{-}$, the complex of free $\cS_{\BN}$-modules
    that computes $\Inat(K_{2,3}^{-}; \Gamma_{\BN})$ is chain-homotopy equivalent
    to the complex
    \[
              \cS_{\BN} \oplus \cS_{\BN}\stackrel{\partial}{\to} \cS_{\BN} 
    \]
    where $\partial$ has matrix entries $(L, P)$. In particular,  $\Inat(K_{2,3}^{-};
    \Gamma_{\BN})$ is isomorphic to
    $\cS_{\BN} \oplus (\cS_{\BN}/J)$ as an
    $\cS_{\BN}$-module, where $J$ is again $\langle L,P\rangle$. Furthermore, the
   fractional ideal $\znat_{\BN}(K^{-}_{2,3})$ is $\langle 1\rangle$.
\end{proposition}

\begin{proof}
    Except for the identification of the fractional ideal, this
    proposition is obtained by dualizing the previous one. If we write
    $\epsilon_{1}$, $\epsilon_{2}$ for the basis of $\cS_{\BN} \oplus \cS_{\BN}$,
    then the generator of $\cS_{\BN}$ summand in $\Inat(K_{2,3}^{-};
    \Gamma_{\BN})$ is the element $\tau = P \epsilon_{1} + L \epsilon_{2}$ in
    $\ker(\partial)$. The immersed cobordism $\Sigma$ from $U_{1}$ to
    $K_{2,3}$ in the proof of the previous proposition gives a
    cobordism $\Sigma^{\dag}$ from $K_{2,3}^{-}$ to $U_{1}$, and
    the dual of the previous calculation says that $\Sigma^{\dag}$ acts as
     \[
     \begin{aligned}
         \epsilon_{2} &\mapsto \alpha \\
         \epsilon_{1} &\mapsto 0.   
     \end{aligned}
      \]
     So $\Sigma^{\dag}$ maps the generator $\tau$ of
     $\Inat(K_{2,3}^{-} ;\Gamma_{\BN})/\mathrm{torsion}$ to
     $L\alpha$. That is, the image of the map
     \[
                \Inat(\Sigma^{\dag}; \Gamma_{\BN}) : \Inat(K_{2,3}^{-}
                ;\Gamma_{\BN}) \to \Inat(U_{1} ;\Gamma_{\BN}) \cong \cS_{\BN}
     \]
     is the ideal $\cI = L\cS_{\BN}$. The immersed cobordism $\Sigma^{\dag}$
     has one positive double point, so by the characterization in
     Lemma~\ref{lem:z-backwards} we have
     \[
          \begin{aligned}
             \znat_{\BN}(K_{2,3}^{-}) &= L^{-1} \cI \\
                                            &= \langle 1 \rangle.
          \end{aligned}
     \]  
     This completes the proof of Proposition~\ref{prop:trefoil-minus}.     
\end{proof}

Having identified the complexes involved, it is a straightforward
matter to apply a further change of basis, $\cS_{\BN} \to \cS$, where
$\cS$ is a valuation ring, so that we may compute the real-valued invariants
$\scrf_{r}(K_{2,3})$ and $\scrf_{r}(K^{-}_{2,3})$ for the two
trefoils, $0\le r \le 1$. (See
Example B in section~\ref{subsec:Examples}.) We obtain
\[
\begin{aligned}
    \scrf_{r}(K_{2,3}) &= r \\
    \scrf_{r}(K^{-}_{2,3}) &= -r.
\end{aligned}
\]
To illustrate the calculation in the case of $K^{-}_{2,3}$, following
the line in Proposition~\ref{prop:trefoil-minus}, the complex is now
\[
  \cS \oplus \cS\stackrel{\partial}{\to} \cS
\]
where
\[
\begin{aligned}
    \partial(\epsilon_{1}) &= u_{1} x^{4r}  \\
    \partial(\epsilon_{2}) &= u_{2} x^{4}
\end{aligned}
\]
where $u_{1}$ and $u_{2}$ are units. The free summand of the homology
(the kernel of $\partial$) is generated now by $\tau = u_{1}^{-1}u_{2} x^{4-4r}
\epsilon_{1} + \epsilon_{2}$. The map arising from $\Sigma^{\dag}$ in
the proof of Proposition~\ref{prop:trefoil-minus} maps $\tau$ to
$\alpha \in \cS$, the generator. Therefore
$\znat_{\BN}(K^{-}_{2,3}) = L^{-1}\cS$ which is the ideal
$\langle x^{-4r}\rangle$. The invariant $\scrf_{r}(K^{-}_{2,3})$ is
the order of this ideal, which is $-r$, because the order of the ideal
$\langle x\rangle$ is $1/4$ by convention \eqref{eq:quarter}.

\section{Further calculations}
\label{subsec:further}

As mentioned in an earlier remark, it is possible to obtain a
description of $\Inat(K;\Gamma_{\BN})$ for the trefoil based only on
the formal properties of instanton homology. In fact one can extend
such arguments a little further. For example, in the case of the
2-stranded torus knot $K_{2,2\ell+1}$, the ideal
$\znat_{\BN}(K_{2,2\ell+})$ can be shown to be $J^{\ell}$,
generalizing the result for the trefoil ($\ell=1$). Indeed, the
complex that computes $\Inat(K_{2,2\ell+1};\Gamma_{\BN})$ can be
characterized uniquely up to chain-homotopy equivalence. Similar
calculations can be made for the 2-component torus links $K_{2,2\ell}$, for the
twist knots, and for some small pretzel
knots. 

In all these simple cases, the results which are obtained
coincide with the results for Heegaard Floer homology, in the version
set up in \cite{Alishahi-Eftekhary}. In particular, in the notation of
\cite{Alishahi-Eftekhary}, the Heegaard Floer complex is a complex of
free modules over the ring $\F_{2}[\mathsf{u},\mathsf{w}]$, and the
instanton chain complex is obtained by making the base-change
\[
\mathsf{u}\mapsto L, \qquad \mathsf{w}\mapsto P.
\]
This coincidence for such simple knots is an inevitable consequence of
the formal properties that the two theories share, and it is not clear
whether it extends much further. In \cite{Alishahi-Eftekhary}, there
is a complete symmetry between the variables $\mathsf{u}$ and
$\mathsf{w}$, a symmetry which is also reflected in formal properties
of the closely-related invariant $\Upsilon_{K}(t)$ from
\cite{Upsilon}. If such a symmetry is present in the instanton theory,
then it is not apparent on the surface.

An example where a divergence between the instanton theory and the
Heegaard-Floer theory might be apparent is the torus knot
$K_{3,4}$. Based on preliminary calculations, the authors conjecture
that the ideal $\znat_{\BN}(K_{3,4})$ is given by
\[
        \znat_{\BN}(K_{3,4}) = \langle L^{3}, L^{2} P, L P^{2}, P^{3},
        Y \rangle
\]
where 
\[
\begin{aligned}
    Y 
      &= (1+T^{-2}) P^{2} + L^{2}.
\end{aligned}
\]
On the other hand, the ideal $\mathbb{A}(K_{3,4})$ from
\cite{Alishahi-Eftekhary}, based on the calculation of the Heegaard
Floer homology of torus knots from \cite{OS-torus}, is
\[
\begin{aligned}
    \mathbb{A}(K_{3,4}) &= \langle \mathsf{u}^{3}, 
\mathsf{u}^{2} \mathsf{w}, \mathsf{u} \mathsf{w}^{2}, \mathsf{w}^{3}, Z
    \rangle \\
   &= \langle \mathsf{u}^{3}, \mathsf{w}^{3}, Z
    \rangle,
\end{aligned}
\]
where $Z=\mathsf{u}\mathsf{w}$. In particular, while
$\mathsf{u}\mathsf{w}$ belongs to the ideal in the Heegaard-Floer
case, the conjectural calculation implies that $LP$ does not belong to
the ideal $\znat_{\BN}(K_{3,4})$. The calculation for $K_{3,4}$
can be extended to the other torus knots $K_{3,p}$, and the authors
hope to return to these and other questions in a future paper.

\bibliographystyle{abbrv}
\bibliography{ibn}

\end{document}